\documentclass[reqno]{amsart}

\usepackage{graphicx}
\usepackage{amsmath}
\usepackage{multicol}
\usepackage{mathrsfs}
\usepackage{amssymb}
\usepackage{amsthm}
\usepackage{mathtools}
\usepackage[onehalfspacing]{setspace}
\usepackage[top=90pt,bottom=80pt,left=80pt,right=80pt,footskip=30pt]{geometry}
\usepackage{xcolor}
\usepackage{enumerate}
\usepackage[all]{xy}
\usepackage{hyperref}
\usepackage{tikz}
\usetikzlibrary{calc}
\usepackage{enumitem}

\theoremstyle{definition}
\newtheorem*{Def}{Definition}
\newtheorem{Thm}{Theorem}[section]
\newtheorem{Prop}[Thm]{Proposition}
\newtheorem{Cor}[Thm]{Corollary}
\newtheorem{Lem}[Thm]{Lemma}
\newtheorem{Ex}[Thm]{Example}
\newtheorem{Rmk}[Thm]{Remark}
\newtheorem*{Main0}{Theorem}
\newtheorem*{Main1}{Theorem \ref{higher}}
\newtheorem*{Main1.5}{Theorem \ref{main}}
\newtheorem*{Main2}{Theorem \ref{classification_rank_one} and \ref{classification_rank_two}}
\newtheorem*{Main3}{Corollary \ref{finite}}
\newtheorem*{Main4}{Remark \ref{etale}}

\newcommand{\PP}{\mathbb{P}}
\newcommand{\CC}{\mathbb{C}}
\newcommand{\RR}{\mathbb{R}}
\newcommand{\ZZ}{\mathbb{Z}}

\newcommand{\Oo}{\mathcal{O}}
\newcommand{\Ii}{\mathcal{I}}
\newcommand{\SU}{\mathcal{SU}}
\renewcommand{\aa}{\mathfrak{a}}
\newcommand{\bb}{\mathfrak{b}}
\newcommand{\cc}{\mathfrak{c}}
\newcommand{\dd}{\mathfrak{d}}
\newcommand{\ee}{\mathfrak{e}}
\newcommand{\morph}{\gamma}
\newcommand{\Div}{\operatorname{Div}}
\newcommand{\Pic}{\operatorname{Pic}}
\newcommand{\NE}{\operatorname{NE}}
\newcommand{\rk}{\operatorname{rk}}
\newcommand{\Hom}{\operatorname{Hom}}

\newcommand{\im}{\operatorname{im}}
\newcommand{\len}{\operatorname{len}}
\newcommand{\elm}{\mathrm{elm}}
\newcommand{\Nm}{\mathrm{Nm}}
\newcommand{\Bl}{\mathrm{Bl}}
\renewcommand{\Pr}{\mathrm{Pr}}
\makeatletter
\providecommand{\leadsfrom}{%
    \mathrel{\mathpalette\reflect@squig\relax}%
}
\newcommand{\reflect@squig}[2]{%
    \reflectbox{$\m@th#1\leadsto$}%
}
\makeatother

\title{Stability of Symmetric Powers of Vector Bundles\\ of Rank Two with Even Degree on a Curve}
\author{Jeong-Seop Kim}
\address{Department of Mathematical Sciences, KAIST, 291 Daehak-ro, Yuseong-gu, Daejeon 34141, Korea}
\email{jeongseop@kaist.ac.kr}

\begin{document}
\setlength\abovedisplayskip{3.5pt}
\setlength\belowdisplayskip{3.5pt}

\pagestyle{plain}

\maketitle
\vspace{-2em}
\begin{abstract}
This paper treats the strict semi-stability of the symmetric powers $S^k E$ of a stable vector bundle $E$ of rank $2$ with even degree on a smooth projective curve $C$ of genus $g\geq 2$.
The strict semi-stability of $S^2 E$ is equivalent to the orthogonality of $E$ or the existence of a bisection on the ruled surface $\PP_C(E)$ whose self-intersection number is zero.
A relation between the two interpretations is investigated in this paper through elementary transformations.
This paper also gives a classification of $E$ with strictly semi-stable $S^3 E$.
Moreover, it is shown that when $S^2 E$ is stable, every symmetric power $S^k E$ is stable for all but a finite number of $E$ in the moduli of stable vector bundles of rank $2$ with fixed determinant of even degree on $C$.
\end{abstract}

\section{Introduction}

All varieties are defined over the field of complex numbers $\CC$.
Let $C$ be a smooth projective curve and $E$ be a vector bundle on $C$ of rank $r=\rk E$ and degree $d=\deg E$.
Then $E$ is said to be \emph{stable} (resp. \emph{semi-stable}) if, for every nonzero proper subbundle $F$ of $E$ with torsion-free quotient $E/F$, the slope $\mu(F)$ of $F$ is less than (resp. less than or equal to) $\mu(E)$ where the \emph{slope} is defined by $\mu(F)={\deg F}/{\rk F}$ \cite[p.~87]{Friedman98}.
We may assume that the quotients are locally free since a torsion-free sheaf is locally free on $C$.\linebreak
A semi-stable vector bundle $E$ is called \emph{strictly semi-stable} if $E$ is not stable.

It is natural to investigate the stability of the symmetric power $S^kE$ of $E$ when $E$ is stable.
Using the correspondence between the stability of $E$ and the irreducibility of its associated unitary representation $\rho_E: \pi_1(C)\to U(r)$ identified by Narasimhan and Seshadri \cite{Narasimhan-Seshadri65}, it is possible to prove the following.

\begin{Main0}[{\cite[p.~53]{Hartshorne70}}]
Let $C$ be a smooth projective curve of genus $g\geq 2$ and $E$ be a stable vector bundle\linebreak on $C$.
Then
\begin{enumerate}
\item $S^k E$ is semi-stable for every $k\geq 2$ for all $E$, and
\item $S^k E$ is stable for every $k\geq 2$ for sufficiently general $E$.
\end{enumerate}
\end{Main0}

The next question would be a classification of stable $E$ for which $S^k E$ is strictly semi-stable.
Because we are interested in the case where $r=2$ and $d$ is even, we may assume that $\det E=\Oo_C$ after substituting $E$ by $E\otimes L^{-1}$ for some line bundle $L$ with $L^2=\det E$.
Then $E$ is self-dual as $E\cong E^\vee\otimes (\det E)\cong E^\vee$.
Note that $S^k(E\otimes L^{-1})=S^kE\otimes L^{-k}$ and the stability is invariant under the twist by a line bundle.
As $\det S^k E=(\det E)^{\binom{k+1}{2}}=\Oo_C$ and $\deg S^k E=0$, $S^k E$ is not stable if it has a proper nonzero subbundle $V\to S^k E$ of $\deg V=0$ with the quotient vector bundle $Q$ of $\deg Q=0$.
By taking the dual of the quotient\linebreak $S^k E\to Q$ together with the self-duality of $S^k E$ given by
\[
(S^k E)^\vee
\cong S^k E^\vee
\cong S^k E,
\]
we get a subbundle $Q^\vee\to S^k E$ of $\deg Q^\vee=-\deg Q=0$ and $\rk V+\rk Q^\vee =\rk V+\rk Q=\rk S^k E=k+1$.
Thus $S^k E$ has destabilizing subbundles $V$ and $Q^\vee$, one of which has rank less than or equal to $\frac{k+1}{2}$ (see also {\cite[Proposition 2.1]{Choe-Choi-Kim-Park18}}). 

\begin{Rmk}\label{first}
Let $E$ be a stable vector bundle on $C$ of rank $2$ with even degree.
Then $S^k E$ is strictly semi-stable if and only if it is destabilized by a subbundle $V\to S^k E$ of $\rk V\leq \frac{k+1}{2}$.
Similarly, it suffices to consider the quotient bundles $S^k E\to Q$ of $\rk Q\leq \frac{k+1}{2}$ to determine the stability of $S^k E$.
\end{Rmk}

By the remark and assuming the stability of lower symmetric powers, we have the following reduction.

\begin{Main1}
Let $C$ be a smooth projective curve of genus $g\geq 2$ and $E$ be a stable vector bundle on $C$\linebreak of rank $2$ with even degree.
Then $S^k E$ is stable for all $k\geq 2$ unless one of the following cases occurs.
\begin{enumerate}
\item $S^2 E$ is destabilized by a vector bundle of rank $1$
\item $S^3 E$ is destabilized by a vector bundle of rank $2$
\item $S^4 E$ is destabilized by a vector bundle of rank $2$
\item $S^6 E$ is destabilized by a vector bundle of rank $3$
\end{enumerate}
In particular, $S^k E$ is stable for every $k\geq 2$ if $S^m E$ is stable for all $m\leq 6$.
\end{Main1}

We will classify cases (1), (2) in Section 2, 3, and give a proof for the following result in Section 4.

\begin{Main1.5}
If $S^2 E$ is stable, then $S^k E$ is stable for every $k\geq 3$ except for finitely many $E$ in the moduli of stable vector bundles on $C$ of rank $2$ with fixed determinant of even degree.
\end{Main1.5}

For $k=2$, $S^2 E$ is strictly semi-stable if and only if it is destabilized by a quotient line bundle $S^2 E\to L$ of $\deg L=0$.
We will observe that this is equivalent to saying that $E$ is \emph{orthogonal} where\linebreak an orthogonal bundle is defined by a vector bundle $E$ with a nondegenerate symmetric bilinear form $E\otimes E\to L$ for some line bundle $L$.
According to Mumford's classification \cite{Mumford71}, the orthogonal bundles of rank $2$ are given by the direct images of line bundles on an unramified double covering $B\to C$.\linebreak
We will see that they form a positive dimensional family with fixed determinant.

On the other hand, Choi and Park \cite{Choi-Park15} show that there exists a stable $E$ whose associated ruled surface $\PP_C(E)$ admits a bisection of zero self-intersection using elementary transformations.
Then its symmetric square $S^2 E$ is strictly semi-stable due to the following correspondence.
\[
\left\{\text{$k$-sections $D$ on $\PP_C(E)$ with $D^2=2kb$}\right\}
~\leftrightarrow~
\left\{\text{line subbundles $L^{-1}\to S^k E$ with $\deg L=b$}\right\}
\]
In Section $2$, we will see a relation between orthogonal bundles and such ruled surfaces,
and find that every orthogonal bundle can be obtained by the method of Choi and Park \cite{Choi-Park15}.

For $k=3$, since $S^3 E$ is of rank $4$, if not stable, it is destabilized by a subbundle of rank $1$ or $2$.

\begin{Main2}
Let $C$ be a smooth projective curve of genus $g\geq 2$ and $E$ be a stable vector bundle on $C$ of rank $2$ with trivial determinant.
Then
\begin{enumerate}
\item  $S^3 E$ is destabilized by a subbundle of rank $1$ only if $S^2 E$ is strictly semi-stable, and there exist only a finite number of such $E$,
\item $S^3 E$ is destabilized by a subbundle of rank $2$ if $S^2 E$ is strictly semi-stable, and there are only finitely many such $E$ with stable $S^2 E$.
\end{enumerate}
In particular, except a finite number of $E$, $S^3 E$ is strictly semi-stable if and only if $S^2 E$ is not stable.
\end{Main2}

In Section $3$, we will classify the exceptional cases as the ones
 satisfying $S^2 E=\eta_*R$ for some unramified cyclic $3$-covering $\eta: C'\to C$ and $R\in J^0(C')\backslash \eta^*J^0(C)$ with $R^2=\Oo_{C'}$.
It also completes the description of $E$ with $S^4 E$ being destabilized by a line subbundle (see Proposition \ref{tosmall}).

For $k\geq 4$, the remaining cases for the stability of $S^k E$ are $k=4$ and $6$ as stated in Theorem \ref{higher}.
In Section 4, we will show that each case is further reduced to the case where $S^l E$ is destabilized by a line subbundle for some $l\geq k$.
Then, with the aid of the following corollary, we obtain the result that there are\linebreak only finitely many $E$ with trivial determinant where $S^2 E$ is stable but $S^k E$ is not stable for some $k\geq 3$.

\begin{Main3}
Let $C$ be a smooth projective curve of genus $g\geq 2$ and $E$ be a stable vector bundle on $C$\linebreak of rank $2$ with trivial determinant.
If $k\geq 3$, then there exist at most finitely many $E\in\SU_C(2,\Oo_C)$ where\linebreak $S^k E$ is destabilized by a line subbundle but $S^m E$ is not destabilized by a line subbundle for any $m<k$.
\end{Main3}

Under the correspondence between the line subbundles of $S^k E$ and the $k$-sections on $X=\PP_C(E)$, the result says that if $S^2 E$ is stable and $E$ has even degree, then there are only a finite number of $X$ which admits a $k$-section of zero self-intersection for some $k\geq 3$.
Notice that the class of $k$-secant divisors of zero self-intersection lies on the boundary of the closure of the cone of curves $\overline{\NE}(X)$ in $N_1(X)$ when $E$ is stable \cite[I:~p.~70]{Lazarsfeld04}.
Thus the result also answers to the question of how many $X=\PP_C(E)$ has closed $\NE(X)$ for the stable vector bundles $E$ of rank $2$ with even degree when $S^2 E$ is stable.

Due to an observation by Rosoff \cite[p.~123]{Rosoff02}, when $D$ is the $k$-section on $X=\PP_C(E)$ corresponding to a destabilizing line subbundle of $S^k E$, the induced $k$-covering $\pi: D\to C$ is necessarily unramified.
Also, the covering gives a destabilizing quotient line bundle $\pi^* E\to R$ so that $\pi^*E$ is not stable on $D$.\linebreak
Further, we can derive a stronger assertion that $\pi^*E$ is not only strictly semi-stable but it also splits into the direct sum of line bundles, and the line bundles are torsion elements in the Picard group when $k\geq 3$.
By composing a cyclic covering over which trivializes the torsion line bundles, we can conclude that $E$ is trivialized over an unramified finite covering of $C$.
Namely, $E$ satisfies the property called \emph{\'{e}tale triviality} \cite{Biswas20}, which is known to be equivalent to saying that $E$ is a \emph{finite} bundle introduced by Nori \cite{Nori76}.

\begin{Main4}
Let $E$ be a vector bundle on $C$ of rank $2$ with trivial determinant.
Assume that $S^2 E$ is stable.
Then $E$ is finite if and only if $S^k E$ is not stable for some $k\geq 3$.
\end{Main4}

\noindent
{\bf Acknowledgement.}
I would like to thank my thesis advisor, Prof. Yongnam Lee, for introducing this topic and giving valuable guidance.
I am also indebted to an anonymous reviewer of an earlier manuscript\linebreak not only for providing helpful suggestions to improve the presentation but also for pointing out a flaw in the previous proof of Theorem \ref{higher}.
This work will be part of my Ph.D. thesis.
I was partly supported by Samsung Science and Technology Foundation under Project Number SSTF-BA1701-04.
\vspace{1em}

\section{Semi-Stable Vector Bundles Whose Symmetric Square is Not Stable} 

\subsection{Unramified Finite Coverings and Prym Varieties}
As mentioned in the introduction, the strict semi-stability of $S^k E$ has a lot to do with the unramified $k$-coverings of $C$.
In this subsection, we review the theory of unramified finite coverings.
The references are \cite[Exercise III.10.3 \& Exercise IV.2.6]{Hartshorne77} and \cite[Chapter 12]{Birkenhake-Lange04}.

Let $C$ be a smooth projective curve.
We will denote by
\begin{itemize}
\item $\Pic(C)$ the Picard group of $C$,
\item $J^n(C)$ the Jacobian of line bundles on $C$ of degree $n$, and
\item $J_m(C)$ the set of line bundles on $C$ of order $m$; the elements $L\in J^0(C)$ with $L^m=\Oo_C$.
\end{itemize}
We will also abuse notation to take a divisor $\bb$ from the group of divisors as $\bb\in\Pic(C)$ instead of $\Div(C)$.\pagebreak

\newgeometry{top=90pt,bottom=80pt,left=80pt,right=80pt,footskip=30pt}
Let $D$ be an unramified $k$-covering of $C$, that is, there is a finite surjective morphism $\pi: D\to C$ of degree $k$ whose ramification divisor is empty.
Then $\pi^*:\Pic(C)\to \Pic(D)$ induces $\pi^*:J^n(C)\to J^{kn}(D)$ since
$\deg\pi^*L=k\deg L$ for $L\in \Pic(C)$.
As $\pi$ is unramified, we have $\omega_D=\pi^*\omega_C$.
That is, $\omega_{D/C}=\Oo_D$.
Moreover, $\pi^*$ has nontrivial kernel; $\pi^*L=\Oo_D$ for some $L\in J^0(C)\backslash\{\Oo_C\}$ if and only if $\pi$ is factored as $\pi: D\to C'\xrightarrow{\eta}C$ for some unramified cyclic covering $\eta:C'\to C$ satisfying $\eta^*L=\Oo_{C'}$ \cite[Proposition 11.4.3]{Birkenhake-Lange04}.
Recall that a torsion line bundle $L\in J_m(C)$ defines an unramified cyclic $m$-covering $\eta:C'\to C$ and vice versa under the relations $\eta^*L=\Oo_{C'}$ and $\eta_*\Oo_{C'}=\Oo_C\oplus L^{-1}\oplus \cdots\oplus L^{-(m-1)}$.

To a finite covering $\pi:D\to C$ of degree $k$, we associate the \emph{Norm map} $\Nm_{D/C}: \Pic(D)\to \Pic(C)$ which is defined by
\[
\Nm_{D/C}\left(\sum a_ix_i\right)=\sum a_i\pi(x_i)
\]
in terms of Weil divisors.
Then $\Nm_{D/C}$ is a group homomorphism and $\Nm_{D/C}\circ\pi^*: \Pic(C)\to \Pic(C)$ is multiplication by $k$.
For $R\in \Pic(D)$, $\pi_*R$ is a vector bundle of rank $k$ whose determinant is given by
\[
\det \pi_*R\cong \det \pi_*\Oo_{D}\otimes \Nm_{D/C}(R).
\]
Notice that $(\det \pi_*\Oo_D)^2=\Oo_C(-B)$ for the branch divisor $B$ of $\pi: D\to C$.
Thus, if $\pi$ is unramified, then $(\det \pi_*\Oo_D)^2=\Oo_C$.

If $V$ is a vector bundle on $C$, then $\pi^*V$ is semi-stable if and only if $V$ is semi-stable \cite[II: Lemma 6.4.12]{Lazarsfeld04}.
Meanwhile, if $W$ is a vector bundle on $D$, then it is known that $\pi_*W$ is stable for general $W$ when 
$\pi$ is unramified \cite{Beauville00}.
The following proposition tells that $\pi_*W$ is semi-stable and which $\pi_*W$ is stable in the case where $\deg \pi=2$ and $\rk W=1$.
Notice that any unramified double covering is a cyclic covering.

\begin{Prop}\label{img}
Let $\pi: B\to C$ be a nontrivial unramified double covering corresponding to $M\in J_2(C)$ and $R\in J^0(B)$.
Then $E=\pi_*R$ is semi-stable and is strictly semi-stable if and only if $R\in \pi^*J^0(C)$.
Moreover, $E$ splits as $E=L\oplus (L\otimes M)$ for some $L\in J^0(C)$ when it is strictly semi-stable.
\end{Prop}

\begin{proof}
Note that the natural morphism $\pi^*\pi_*R\to R$ becomes surjective because $\pi$ is an affine morphism.
Since the kernel of a surjection between vector bundles is a vector bundle, the sequence 
\[
0\to K\to \pi^* E\to R\to 0
\]
is exact for some vector bundle $K$ on $B$ where $\rk K=1$ as $\rk \pi^*\pi_*R=\rk \pi_*R=2$ and $\rk R=1$ in this case.
By comparing the determinants using $\det\pi^*E=\pi^*\det E$, the exact sequence becomes
\[
0\to R^{-1}\otimes\pi^*\det E\to \pi^*E\to R\to 0.
\]
Then, from
\begin{align*}
\deg E
=\deg (\det\pi_*R)
=\deg (\det\pi_*\Oo_B \otimes \Nm_{D/C}(R))
=\deg (\det(\Oo_C\oplus M)\otimes \Nm_{D/C}(R))
=0,
\end{align*}
we have $\deg \pi^*E=2\deg E=0$ and $\deg (R^{-1}\otimes \pi^*\det E)=\deg \pi^*E-\deg R=0$.

If there is an injection $L\to E$ for some line bundle $L$ of degree $d$, then we get a nonzero morphism $\pi^*L\to R$ due to the adjoint property, $\Hom(\pi^*L,R)=\Hom(L,\pi_*R)=\Hom(L,E)$.
As $\deg\pi^*L=2d$, $d\leq 0$ and the equality holds if and only if $R=\pi^*L$.
Therefore, $E$ is semi-stable, and $E$ is strictly semi-stable if and only if $R\in \pi^*J^0(C)$.

Now if $E$ is strictly semi-stable so that $R=\pi^*L$ for some $L\in J^0(C)$, then, by the projection formula,
\[
\pi_*R
=\pi_*\pi^*L
=(\pi_*\Oo_B)\otimes L
=(\Oo_C\oplus M)\otimes L=L\oplus(L\otimes M).\qedhere
\]
\end{proof}
\restoregeometry

In the case where $\pi: B\to C$ is a nontrivial unramified double covering, we denote the kernel of $\Nm_{B/C}$ by $\Pr(B/C)$ in this paper.
Then $\Pr(B/C)\subseteq J^0(C)$ and it has two components as
\[
\Pr(B/C)
=\{S\otimes (\iota^*S)^{-1}\,|\,S\in \Pic(B)\}
=\{S\otimes (\iota^*S)^{-1}\,|\,S\in J^0(B)\}\cup\{S\otimes (\iota^*S)^{-1}\,|\,S\in J^1(B)\}
\]
where $\iota:B\to B$ is the involution induced by $\pi$ \cite[Lemma 1]{Mumford71}.
We denote the first summand by $\Pr^0(B/C)$ and the second one by $\Pr^1(B/C)$.
$\Pr^0(B/C)$ is known as the \emph{Prym variety} of $B$ over $C$ which is an abelian subvariety of $J^0(B)$ of dimension $g(B)-g(C)=g-1$, and $\Pr^1(B/C)$ is a translate of $\Pr^0(C)$ in $J^0(B)$.\linebreak
It is also known that $J^0(B)=\Pr^0(B/C)\otimes \pi^*J^0(C)=\{R\otimes \pi^*L\,|\,R\in\Pr^0(B/C),\,L\in J^0(C)\}$, and we can describe the intersection $\Pr^0(B/C)\cap \pi^*J^0(C)=\Pr^0(B/C)\cap J_2(B)$ using the following proposition.

\begin{Prop}\label{intersection}
Let $\pi: B\to C$ be a nontrivial unramified double covering.
Then
\[
\Pr(B/C)\cap \pi^*J^0(C)=\pi^*J_2(C)=\Pr(B/C)\cap J_2(B).
\]
\end{Prop}

\begin{proof}
As $\Nm_{B/C}(\pi^*L)=L^2$ for $L\in \Pic(C)$, $\pi^*L\in\Pr(B/C)$ if and only if $L\in J_2(C)$.
Thus we get $\Pr(B/C)\cap\pi^*J^0(C)=\pi^*J_2(C)$, and its order is $2^{2g-1}$ because $|J_2(C)|=2^{2g}$ and $|\ker\pi^*|=2$ for the genus $g$ of $C$.
From $\pi^*J_2(C)\subseteq J_2(B)$, we have $\pi^*J_2(C)=\Pr(B/C)\cap\pi^*J_2(C)\subseteq \Pr(B/C)\cap J_2(B)$ and the inclusion becomes equality after calculating the order $|\Pr(B/C)\cap J_2(B)|=2^{2g-1}$.

Since $\Pr^0(B/C)$ is an abelian subvariety of $J^0(B)$ of dimension $g-1$, $|\Pr^0(B/C)\cap J_2(B)|=2^{2(g-1)}$.
So there exists $L_1\in J_2(C)$ with $\pi^*L_1\in \Pr^1(B/C)$ but $\pi^*L_1\not\in\Pr^0(B/C)$.
Then, using the translation $\Pr^1(B/C)=\Pr^0(B/C)\otimes \pi^*L_1$, we can deduce that $|\Pr^1(B/C)\cap J_2(B)|=2^{2(g-1)}$.
Hence we obtain $|\Pr(B/C)\cap J_2(B)|=2^{2g-1}$ and the equality $\pi^*J_2(C)=\Pr(B/C)\cap J_2(B)$.
\end{proof}

\subsection{Classification of Orthogonal Bundles}
The orthogonal bundles are studied by several authors;
Ramanathan \cite{Ramanathan96},
Mumford \cite{Mumford71},
Ramanan \cite{Ramanan81}, Hitching \cite{Hitching07}, and Biswas-G\'{o}mez \cite{Biswas-Gomez10} for instance.\linebreak
In this paper, we use the following definition presented in Hitching \cite{Hitching07}, which is also similar to that given in Biswas-G\'{o}mez \cite{Biswas-Gomez10}.

\begin{Def}
Let $E$ be a vector bundle and $M$ be a line bundle on $C$.
Then $E$ is said to be \emph{orthogonal with values in $M$} if there is a nondegenerate symmetric bilinear form $E\otimes E\to M$.
\end{Def}

Let $E$ be a vector bundle on $C$ of rank $2$ and degree $0$.
If $S^2 E$ is strictly semi-stable, then there is a quotient line bundle $S^2 E\to M$ of $\deg M=0$, which gives a nonzero morphism $E\to E^\vee\otimes M$ from
\[
\Hom(S^2 E, M)
\subseteq \Hom(E\otimes E,M)
\cong H^0(E^\vee\otimes E^\vee\otimes M)
\cong \Hom(E,E^\vee\otimes M).
\]
If $E$ is stable, then the morphism is necessarily an isomorphism, and the induced symmetric bilinear form $E\otimes E\to M$ is nondegenerate on each fiber.
Hence $E$ admits an orthogonal structure.

Conversely, if $E$ is an orthogonal bundle with values in $M$, then it associates a morphism $S^2 E\to M$ which must be surjective because the form $E\otimes E\to M$ is nondegenerate.

\begin{Rmk}\label{twotor}
Let $E$ be a stable vector bundle on $C$ of rank $2$ and degree $0$.
Then $E$ is orthogonal if and only if $S^2 E$ is strictly semi-stable.
If $S^2 E$ is destabilized by a quotient line bundle $S^2 E\to M$, then $E$ is orthogonal with values in $M$.
Also, by comparing the determinants in the isomorphism $E\cong E^\vee\otimes M$, we have $M^2\cong (\det E)^2$.
On the other hand, if $S^2 E$ is destabilized by a line subbundle $M^{-1}\to S^2 E$, then $M^2=(\det E)^{-2}$ follows from the isomorphism $E^\vee\cong E\otimes M$. In particular, if $E$ is orthogonal with values in $M$ and $E$ has trivial determinant, then $M\in J_2(C)$.
\end{Rmk}

\newgeometry{top=75pt,bottom=80pt,left=80pt,right=80pt,footskip=30pt}
By Mumford's classification \cite{Mumford71}, if $E$ is an orthogonal bundle of rank $2$ with values in $\Oo_C$, then
\begin{enumerate}
\item[(1)] $E=L^{-1}\oplus L$ for some line bundle $L$, or
\item[(2)] $E=\pi_*R$ where $\pi: B\to C$ is an unramified double covering and $R$ is a line bundle on $B$ such that $\Nm_{B/C}(R)=\Oo_C$.
\end{enumerate}
In (2), if $\pi$ is the trivial double covering, then $E=\pi_*R$ is the direct sum of line bundles of degree~$0$, and it has trivial determinant as $\det \pi_*\Oo_B=\det(\Oo_C\oplus \Oo_C)=\Oo_C$.
So $E=L^{-1}\oplus L$ for some $L\in J^0(C)$.
This case is covered by (1).
Otherwise, if $\pi$ is a nontrivial double covering corresponding to $M\in J_2(C)\backslash\{\Oo_C\}$, then we know from Proposition~\ref{img} that $E=\pi_*R$ is semi-stable, and it is strictly semi-stable if and only if $E=L\oplus (L\otimes M)$ for some $L\in J^0(C)$.
In this case, $L$ must satisfy $L^2=\Oo_C$ as $\det E=M$.

\begin{Rmk}\label{sss}
If $E$ is strictly semi-stable, then $E$ is orthogonal with values in $\Oo_C$ if and only if
\begin{itemize}
\item $E=L^{-1}\oplus L$ if the determinant of $E$ is trivial,
\item $E=L\oplus (L\otimes M)$ for some $L\in J_2(C)$ if $M=\det E$ is nontrivial.
\end{itemize}
In particular, there is no stable $E$ with trivial determinant whose orthogonal form takes its values in $\Oo_C$.

If $E$ is strictly semi-stable, then $S^2 E$ is always strictly semi-stable because a surjection $E\to L$ induces a surjection $S^2 E\to L^2$.
In the same way, if $E$ is not stable, then $S^k E$ is not stable for any $k\geq 2$.
\end{Rmk}

We denote by $\SU_C(2,\Oo_C)$ the space of S-equivalence classes of semi-stable vector bundles on $C$ of\linebreak rank $2$ with trivial determinant.
After choices of a nontrivial unramified double covering $\pi:\,B\to C$ corresponding to $M\in J_2(C)\backslash\{\Oo_C\}$ and a line bundle $A$ satisfying $A^2=M$, we can define a map $\Phi_{A}:\,\Pr(B/C)\to \SU_C(2,\Oo_C);\,R\mapsto \pi_*R\otimes A$ as $\pi_*R$ is semi-stable by Proposition \ref{img}.
For fixed $M$,\linebreak the images of $\Phi_A$ are the same under the changes of $A$ because the choices differ by a twist of $L\in J_2(C)$ and $\Pr(B/C)$ is invariant under the translation by $\pi^*L$ if $L\in J_2(C)$.

\begin{Prop}\label{dim}
If $E\in \SU_C(2,\Oo_C)$ is stable and $S^2 E$ is strictly semi-stable, then $E=\Phi_{A}(R)=\pi_*R\otimes A=\pi_*(R\otimes \pi^*A)$ for some $\pi: B\to C$, $A\in J^0(C)$, and $R\in\Pr(B/C)$ as above.
Moreover, the locus of $E\in \SU_C(2,\Oo_C)$ where $E$ is stable and $S^2 E$ is strictly semi-stable has dimension $g-1$.
\end{Prop}

\begin{proof}
Due to Remark \ref{twotor}, if $S^2 E$ is not stable, then $E$ is orthogonal with values in $M\in J_2(C)\backslash\{\Oo_C\}$.
Thus for any line bundle $A$ with $A^2=M$, $E\otimes A$ becomes an orthogonal bundle with values in $\Oo_C$.
Then, by Mumford's classification, $E\otimes A=\pi_*R$ for some nontrivial unramified double covering $\pi:B\to C$ and $R\in\Pr(B/C)$ where $\pi$ corresponds to $M$ since $\det\pi_*\Oo_B=\det\pi_*R=A^2=M$.

In order to prove the next claim, it is enough to show that $\pi_*:\Pr(B/C)\to \SU_C(2,M)$ is generically $2$-to-$1$ for a fixed nontrivial unramified double covering $\pi:B\to C$ corresponding to $M\in J_2(C)\backslash\{\Oo_C\}$ because $\dim\Pr(B/C)=g-1$ and the choice of $M$ is finite.
Let $R,\,R'\in\Pr(B/C)$ and $E=\pi_*R=\pi_*R'$.
As $\pi^*\det E=\pi^*M=\Oo_B$, we get the following exact sequences on $B$ (see the proof of Proposition \ref{img}).
\[
0\to R^{-1}\to \pi^*E\to R\to 0
\quad\text{and}\quad
0\to R'^{-1}\to \pi^*E\to R'\to 0
\]
If $R'\neq R$, then there exists a nonzero morphism $R^{-1}\to R'$ by adjoining the morphisms $R^{-1}\to\pi^*E$ and $\pi^*E\to R'$.
Since both $R$ and $R'$ have degree $0$, $R'\cong R^{-1}$.
So we have either $R'=R$ or $R'=R^{-1}$.
Here, $R^{-1}=\iota^*R$ for the involution $\iota:B\to B$ induced by $\pi$ because $R=S\otimes (\iota^* S)^{-1}$ for some $S\in \Pic(B)$.
Note that $R=\iota^*R$ if and only if $R\in \pi^*J^0(C)$, and $\Pr(B/C)\cap\pi^*J^0(C)$ is finite by Proposition \ref{intersection}.
\end{proof}

Recall that the dimension of the moduli $\SU_C(2,\Oo_C)$ is $3g-3$, and the locus of $E\in \SU_C(2,\Oo_C)$ with $S^2 E$ being strictly semi-stable is the union of the above locus and the locus of strictly semi-stable $E$.
The latter locus is given by the image of $J^0(C)\to \SU_C(2,\Oo_C);\,L\mapsto [L^{-1}\oplus L]$, and its dimension is $g$.

\subsection{\emph{k}-Sections on a Ruled Surface}

The material of this and the next subsection is well-known, but we include it for the sake of notational clarity.
Let $E$ be a vector bundle on $C$ of rank $2$.
The projective space bundle $X=\PP_C(E)$ with projection $\Pi: X\to C$ is called a \emph{ruled surface} over $C$.
We choose the convention that $\PP_C(E)$ is regarded as the projective space of lines in the fibers.
By Tsen's theorem, there exists a section of $\Pi$ and it is possible to regard the image as an effective divisor $C\subseteq X$.
The Picard group of $X$ is given by $\Pic(X)=\ZZ C\oplus\Pi^*\Pic(C)=\{kC+\bb f\,|\,k\in\ZZ,\,\bb\in\Pic(C)\}$.
A \emph{$k$-secant divisor} is a divisor $D$ on $X$ linearly equivalent to $kC+\bb f$ for some $\bb\in\Pic(C)$, and is called a \emph{$k$-section} if $D$ is effective.
If $k=1$ (resp. $2$, $3$), then a $k$-section $D$ is said to be a \emph{section} (resp. \emph{bisection}, \emph{trisection}).\linebreak
We will denote linear equivalence by $\sim$ and numerical equivalence by $\equiv$.

We fix a unisecant divisor $C_1$ on $X$ which satisfies $\pi_*\Oo_X(C_1)=E$.
Then $\pi_*\Oo_X(kC_1)=S^k E$ for $k\geq 0$,\linebreak $\pi_*\Oo_X(kC_1)=0$ for $k<0$, and $R^1\pi_*\Oo_X(kC_1)=\pi_*\Oo_X(-(k+2)C_1)^\vee\otimes (\det E)^\vee$.
We can also deduce that\linebreak ${C_1}^2=\deg E$.
There is a correspondence between $k$-sections on $X$ and line subbundles of $S^k E$ given by
\begin{align*}
\text{effective $D\sim kC_1 + \bb f$}
\ \leftrightarrow\ 
D\in H^0(\Oo_X(kC_1+\bb f))
\ \leftrightarrow\ 
s\in H^0(S^k E\otimes L)
\ \leftrightarrow\ 
\text{inclusion $L^{-1}\xrightarrow{s} S^kE$}
\end{align*}
for $L=\Oo_C(\bb)$, and the self-intersection number of $D$ is equal to
\[
D^2=(kC_1+\bb f)^2=k^2{C_1}^2+2k\deg\bb=k^2\deg E+2k\deg L.
\]

For $k=1$, a section $C_0$ is called a \emph{minimal section} if it attains the minimal self-intersection number $D^2$\linebreak among the sections $D\sim C_1+\bb f$ for some $\bb\in \Pic(C)$.
Though the choice of $C_0$ may not be unique, but the number ${C_0}^2$ is uniquely determined by $E$, and it is called the \emph{Segre invariant} $s_1(E)$.
From the definition\linebreak of stability, it is easy to check that $E$ is stable (resp. semi-stable) if and only if ${C_0}^2>0$ (resp. ${C_0}^2 \geq 0$).

Let $E$ be a semi-stable vector bundle on $C$ of rank $2$ and degree $0$.
Then $S^kE$ is a semi-stable vector bundle of degree $0$, and if $L^{-1}\to S^kE$ is a line subbundle of $\deg L=b$, then it must follow that $b\geq 0$.
That is, $D\equiv kC_1+bf$ is effective only if $b\geq 0$.
Thus the cone of curves $\NE(X)\subseteq N_1(X)$ is contained in\linebreak the cone $\mathcal{C}\subseteq N_1(X)$ which is $\RR_{\geq0}$-spanned by the rays $[f]$ and $[C_1]$.
It is further possible to show that $\overline{\NE}(X)=\mathcal{C}$ when $E$ is semi-stable \cite[I:~p.~70]{Lazarsfeld04}.
Therefore, $S^k E$ is destabilized by a line subbundle if there\linebreak exists a $k$-section $D\equiv kC_1$ on $X$ for some $k>0$, which is equivalent to saying that $\NE(X)$ is closed.

\begin{Rmk}\label{interesting}
Let $E$ be a semi-stable vector bundle on $C$ of rank $2$ with even degree.
There are various characterizations of a $k$-section $D$ on $X$ which corresponds to a destabilizing line subbundle $L^{-1}\to S^k E$.
\begin{enumerate}
\item $D$ has zero self-intersection.
\item $D$ lies on the boundary of $\NE(X)$.
\item $\pi=\Pi|_D:D\to C$ is an unramified $k$-covering if $D$ is irreducible and reduced.
\item $D$ is a smooth curve of genus $kg-k+1$ where $g$ is the genus of $C$ if $D$ is irreducible and reduced.
\end{enumerate}
The proof of (3) is introduced in Rosoff \cite[p.~123: the first remark]{Rosoff02}.
By (3), $D$ is smooth, so the other equivalences can be shown using the adjunction and Hurwitz formula.
\end{Rmk}

\begin{Prop}\label{k-2}
Let $E$ be a vector bundle on $C$ of rank $2$ and degree $0$.
Let $X=\PP_C(E)$ be the ruled surface $\Pi: \PP_C(E)\to C$ and $C_1$ be a unisecant divisor on $X$ with $\Pi_*\Oo_X(C_1)\cong E$.
If there is an irreducible and reduced $k$-section $D\sim kC_1+\bb f$ for some $\bb\in\Pic(C)$ of $\deg \bb=0$,
then $\Oo_D((k-2)C_1+\bb f)=\Oo_D$.
\end{Prop}

\begin{proof}
By Remark \ref{interesting}, $D$ is smooth.
Then we can use the adjunction formula,
\[
\omega_D
=\Oo_X(K_X+D)|_D
=\Oo_X((-2C_1+K_C f)+(kC_1+\bb f))|_D,
\]
and it shows that $\omega_{D/C}=\Oo_D((k-2)C_1+\bb f)$.
Again by Remark \ref{interesting}, $D$ is unramified over $C$ so that
$\omega_{D/C}=\Oo_D$.
Thus we have $\Oo_D((k-2)C_1+\bb f)=\Oo_D$.
\end{proof}
\restoregeometry

\subsection{Elementary Transformations}

Let $E$ be a vector bundle on $C$ of rank $2$ and $X=\PP_C(E)$ be the ruled surface $\Pi:\PP_C(E)\to C$.
There are two notions of elementary transformations.
One is for vector bundles and the other is for ruled surfaces.
First, we explain the elementary transformation of vector bundles.
Let $P\in C$ and fix a line $x$ in the fiber $E|_P$.
The elementary transformation $\elm_x E$ of $E$ at $x$ is defined by the following exact sequence.
\[
0\to \elm_x E\to E\xrightarrow{\alpha(x)} \CC_P\to 0
\]
Here, $\alpha(x)$ has the kernel $x$ at the fiber $E|_P\to\CC_P$.
Note that $\det (\elm_x E)=\det E\otimes\Oo_C(-P)$.

Next, for the elementary transformation of ruled surfaces, let $x\in X$ be a closed point.
Notice that the point $x$ in the fiber $Pf=\Pi^{-1}(P)$ can be identified with a line in the fiber $E|_P$ over $P=\Pi(x)$.
Then the elementary transformation $Y:=\elm_x X$ of $X$ at $x$ is the surface given by the following process.
\begin{enumerate}
\item $\widetilde X$ is the blow-up of $X$ at $x$.
The strict transform $\widetilde{Pf}\subseteq\widetilde{X}$ of the fiber $Pf\subseteq X$ is a $(-1)$-curve.
\item $Y$ is the blow-down of $\widetilde{X}$ along $\widetilde{Pf}$.
Then, the strict transform $Z'\subseteq Y$ of the exceptional divisor $Z\subseteq\widetilde{X}$ for $\widetilde{X}\to X$ becomes a smooth rational curve of zero self-intersection.
\end{enumerate}
Then $Y$ is again a ruled surface $\Lambda: \PP_C(F)\to C$ for some vector bundle $F$ on $C$ of rank $2$.
We denote the blow-up and down by
$\varphi:\widetilde{X}\to X$ and $\psi:\widetilde{X}\to Y$, and the center of the blow-down $\psi$ by $y\in Y$.
\[\xymatrix @C=3pc @R=0pc{
&Z\subseteq\widetilde{X}\supseteq \widetilde{Pf}\ar[ld]_\varphi\ar[rd]^\psi&\\
x\in X\ar[rd]_\Pi&&Y\ni y\ar[ld]^\Lambda\\
&P\in C&
}\]

Let $C_1$ be a unisecant divisor on $X$ such that $E=\Pi_*\Oo_X(C_1)$.
Then, for a section $D\sim C_1+\bb f$ on $X$, we have $\Pi_*\Oo_X(D)=E\otimes L$ for $L=\Oo_C(\bb)$.
We will see the nature of the strict transform $D'\subseteq Y$ of $D\subseteq X$ after the elementary transformation with regard to either $x\in D\subseteq X$ or $x\not\in D\subseteq X$.
Let $F=\Lambda_*\Oo_{Y}(D')$.\linebreak
If $x\in D\subseteq X$, then $y\not\in D'\subseteq Y$, and
\begin{align*}
F
&=\Lambda_*\Oo_{Y}(D')
=\Lambda_*(\psi_*(\psi^*\Oo_Y(D')))\\
&=(\Lambda\circ\psi)_*(\Oo_{\widetilde{X}}(\widetilde{D}))
=(\Pi\circ\phi)_*(\Oo_{\widetilde{X}}(\widetilde{D}+Z)\otimes\Oo_{\widetilde{X}}(-Z))\\
&=\Pi_*(\phi_*(\phi^*\Oo_X(D)\otimes\Oo_{\widetilde{X}}(-Z)))
=\Pi_*(\Oo_X(D)\otimes \Ii_x)
\end{align*}
where $\widetilde{D}\subseteq\widetilde{X}$ is the strict transform of $D\subseteq X$.
On the other hand, consider the exact sequence
\[
0\to\Oo_X(D)\otimes \Ii_x\to\Oo_X(D)\to\CC_x\to 0
\]
on $X$.
By pushing forward the sequence, we obtain the following exact sequence on $C$.
\[
0\to F\to E\otimes L\xrightarrow{\beta}\CC_P\to 0
\]
As $\beta=\alpha(x)$, it shows that $F=\elm_x (E\otimes L)=\elm_x E\otimes L$.
Next, if $x\not\in D\subseteq X$, then $y\in D'\subseteq Y$, and
\begin{align*}
F
&=\Lambda_*\Oo_Y(D')
=\Lambda_*(\psi_*(\psi^*\Oo_Y(D')))\\
&=(\Lambda\circ\psi)_*(\Oo_{\widetilde{X}}(\widetilde{D}+\widetilde{Pf}))
=(\Pi\circ\phi)_*(\Oo_{\widetilde{X}}(\widetilde{D}+\widetilde{Pf}+Z)\otimes\Oo_{\widetilde{X}}(-Z))\\
&=\Pi_*(\phi_*(\phi^*\Oo_X(D+Pf)\otimes\Oo_{\widetilde{X}}(-Z)))
=\Pi_*(\Oo_X(D)\otimes \Ii_x)\otimes\Oo_C(P)
\end{align*}
where $\widetilde{Pf}\subseteq \widetilde{X}$ and $\widetilde{D}\subseteq \widetilde{X}$ are as before.
By the same argument, we get the exact sequence
\[
0\to F\otimes\Oo_C(-P)\to E\otimes L\xrightarrow{\beta}\CC_P\to 0
\]
on $C$.
Because $\beta=\alpha(x)$, we deduce that $F=\elm_x(E\otimes L)\otimes\Oo_C(P)=\elm_x E\otimes L(P)$.

\begin{Prop}\label{subbundle_after_elementary_transformation}
Let $D$ be the section on $X=\PP_C(E)$ corresponding to a line subbundle $L^{-1}\to E$.
Then, for the elementary transformation $Y=\elm_x X$ of $X$ at a point $x$ over $P\in C$, and the strict transform $D'\subseteq Y$ of $D\subseteq X$, there exists the corresponding line subbundle
\[
\begin{cases}
\text{$L^{-1}\to \elm_x E$ to $D'\subseteq Y$ and $(D')^2=D^2-1$} & \text{if $x\in D\subseteq X$}, \\
\text{$L^{-1}(-P)\to \elm_x E$ to $D'\subseteq Y$ and $(D')^2=D^2+1$} & \text{if $x\not\in D\subseteq X$}.
\end{cases}
\]
\end{Prop}
\begin{proof}
Note that, in the case where $x\in D$, the strict transform $D'$ on $Y$ corresponds to the line subbundle $L^{-1}\to\elm_xE$.
So the self-intersection number is given by
\[
(D')^2
=\deg(\elm_x E)+2\deg L
=(\deg E-1)+2\deg L
=(\deg E+2\deg L)-1
=D^2-1.
\]
The proof is similar for the case where $x\not\in D$.
\end{proof}

In the both cases, whether $x\in D$ or not, we can see that $\elm_x (\Pi_*\Oo_X(D))$ and $\Lambda_*\Oo_{Y}(D')$ differ by\linebreak a twist of a line bundle, and hence $Y=\PP_C(F)=\PP_C(\elm_x E)$ for $Y=\elm_x X$.

The elementary transformation can be defined at multiple points of $X$ unless the set of points contains distinct points in the same fiber of $\Pi:X\to C$.
We introduce an example of an elementary transformation taken at a double point $x\leadsfrom y$ in which $x$ is a closed point of $X$ over $P\in C$ and $y$ is a point infinitely near to $x$ but not the infinitely near point of the fiber $Pf\subseteq X$.
Equivalently, $y$ is a closed point of $\elm_x X$ over the same point $P\in C$ but not the center of the blow-down $\Bl_x X\to \elm_x X$.

\begin{Ex}
Let $X=\PP_C(E)$ be a ruled surface $\Pi:\PP_C(E)\to C$ and $D$ be a section on $X$ such that $\Pi_*\Oo_X(D)=E\otimes L$ for some line bundle $L$.
Let $x$ be a closed point of $D$ and $y$ be the infinitely near point of $D$ at $x$ given by the intersection point $\widetilde{D}\cap Z$ of the exceptional fiber $Z$ of the blow-up $\text{Bl}_x X\to X$ and the strict transform $\widetilde{D}$ of $D$ on $\text{Bl}_x X$.
Let $P=\Pi(x)$ and $\Ii_{x\leadsfrom y}$ be the ideal sheaf on $X$ which defines $y$ infinitely near to $x$.
By pushing forward the exact sequence
\[
0
\to\Oo_X(D)\otimes \Ii_{x\leadsfrom y}
\to\Oo_X(D)
\to\Oo_X/\Ii_{x\leadsfrom y}
\to 0
\]
on $X$, we have the following exact sequence on $C$.
\[
0\to \elm_{x\leadsfrom y}E\otimes L\to E\otimes L\to \Oo_{2P}\to 0
\]
By Proposition \ref{subbundle_after_elementary_transformation}, we obtain that $\Lambda_*(\Oo_{Y}(D'))=\elm_{x\leadsfrom y}E\otimes L$ where $Y=\elm_{x\leadsfrom y}X$ is a ruled surface $\Lambda:Y\to C$ and $D'\subseteq Y$ is the strict transform of $D\subseteq X$.
\end{Ex}

\subsection{Generation of Orthogonal Bundles by Elementary Transformations}

Choi and Park \cite{Choi-Park15} use elementary transformations to construct a ruled surface $X=\PP_C(E)$ where $E$ is stable and $X$ admits\linebreak a bisection of zero self-intersection.
As we have seen so far, in this case, $S^2 E$ is strictly semi-stable.
Also, $E$ is orthogonal if the degree of $E$ is normalized to be $0$ since $E$ is stable.
In this subsection, we will briefly review the construction, and show that the elementary transformation construction generates all the orthogonal bundles.

Let $M\in J_2(C)$ be nontrivial and $Y$ be the ruled surface $\Lambda: \PP_C(\Oo_C\oplus M)\to C$.
Then $Y$ has only two minimal sections $C_0$ and $C_\infty$ which respectively correspond to $\Oo_C\to\Oo_C\oplus M$ and $M\to\Oo_C\oplus M$.
Because there is a $1$-dimensional family of bisections on $Y$ linearly equivalent to $2C_0$ whereas $Y$ has only finitely many sections numerically equivalent to $C_0$, there exists an irreducible bisection in the linear equivalence class of $2C_0$.
So fix an irreducible bisection $B'\sim 2C_0$ on $Y$ which corresponds to a line subbundle $\Oo_C\to S^2(\Oo_C\oplus M)=\Oo_C\oplus M\oplus\Oo_C$.
Then we can obtain the desired ruled surfaces by taking elementary transformations of $Y$ at general points of $B'$.

Let $x_1,\,x_2,\,\ldots,\,x_{2n}$ be arbitrary closed points of $B'$ on $Y$ and $X$ be the ruled surface $\Pi:X\to C$ obtained by taking elementary transformations of $Y$ at $x_1,\,x_2,\,\ldots,\,x_{2n}$.
To avoid technical issues, we will not deal with the cases where the points involve distinct closed points in the same fiber of $\Lambda$.
However, we allow repeated points.
If $x_1=x_2=\cdots=x_m$ for example, then we can take an elementary transformation at $x_1\leadsfrom x_2\leadsfrom \cdots \leadsfrom x_m$ where $x_{i+1}$ is infinitely near to $x_1$ in the direction $i$-th tangent to $B'$, that is,\linebreak $x_{i+1}$ is given recursively by the intersection point $ \widetilde{B'_i}\cap Z_i$ on the $i$-th blow-up $Y_i\to Y$ where $Z_i$ is the exceptional divisor of the blow-up $Y_i\to Y_{i-1}$ centered at $x_i$ with initial $Y_0=Y$ and $\widetilde{B'_i}\subseteq Y_i$ is the strict transform of $B'\subseteq Y$.

Since ${B'}^2=0$, the smoothness of $B'$ follows from Remark \ref{interesting}.
Then it is easy to check that the strict transform $B\subseteq X$ of $B'\subseteq Y$ satisfies $B^2=0$ as well.
Again by Remark \ref{interesting}, $B$ is smooth, so the strict transform from $B'$ to $B$ is an isomorphism.
Thus we can regard the points $x_1,\,x_2,\,\ldots,\,x_{2n}$ as points of $B$,\linebreak and $\pi=\Pi|_B: B\to C$ as the unramified double covering corresponding to $M$ as the same with $B'\to C$.

\begin{center}
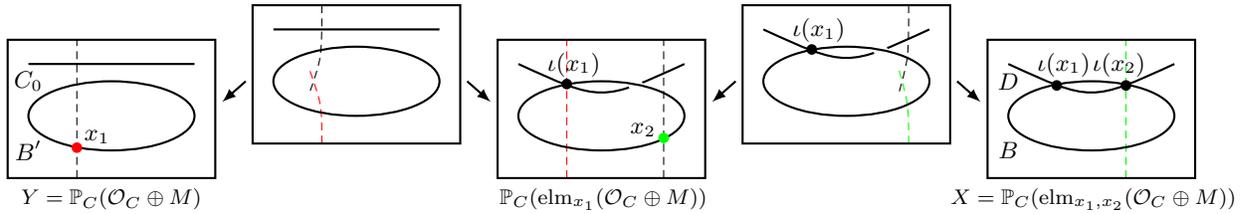
\begin{figure}[h]
\begin{tikzpicture}[scale=0.92]

\pgfmathsetmacro{\wi}{3}
\pgfmathsetmacro{\he}{2}
\pgfmathsetmacro{\ri}{1.18}
\pgfmathsetmacro{\up}{0.5}
\pgfmathsetmacro{\t}{0.1}

\coordinate (a) at (0*\ri*\wi,0);
\coordinate (b) at (1*\ri*\wi,\up);
\coordinate (c) at (2*\ri*\wi,0);
\coordinate (d) at (3*\ri*\wi,\up);
\coordinate (e) at (4*\ri*\wi,0);

\tikzset{>=latex}
\draw[thick,<-]
(0*\ri*\wi+\wi+\t, \he/2+\t) --
(1*\ri*\wi-\t,     \he/2-\t+\up);
\draw[thick,->]
(1*\ri*\wi+\wi+\t, \he/2-\t+\up) --
(2*\ri*\wi-\t,     \he/2+\t);
\draw[thick,<-]
(2*\ri*\wi+\wi+\t, \he/2+\t) --
(3*\ri*\wi-\t,     \he/2-\t+\up);
\draw[thick,->]
(3*\ri*\wi+\wi+\t, \he/2-\t+\up) --
(4*\ri*\wi-\t,     \he/2+\t);

\node at (1.5+0.0*\wi,-0.3)
{\footnotesize $Y=\PP_C(\Oo_C\oplus M)$};
\node at (1.5+2.37*\wi,-0.3)
{\footnotesize $\PP_C(\elm_{x_1}(\Oo_C\oplus M))$};
\node at (1.5+4.73*\wi,-0.3)
{\footnotesize $X=\PP_C(\elm_{x_1,x_2}(\Oo_C\oplus M))$};

\def\surface
{(0,0)--(\wi,0)--(\wi,\he)--(0,\he)--cycle;}
\def\bisection
{(\wi/2,\he/2-0.1) ellipse (1.2 and 0.5)}

\begin{scope}[shift={($(a)$)}]
\draw[thick] {(0.3,\he-0.35)--(\wi-0.3,\he-0.35)};

\draw[black,densely dashed] (\wi/3,0)--(\wi/3,\he);

\draw[thick] \bisection;
\draw[thick] \surface;

\node at (\wi*0.1,\he*0.7) {\small $C_0$};
\node at (\wi*0.1,\he*0.2) {\small $B'$};
\node at (\wi*0.43,\he*0.3) {\small $x_1$};
\node[circle,fill=red,inner sep=0pt,minimum size=4pt] at (\wi/3,\he*0.22) {};
\end{scope}

\begin{scope}[shift={($(b)$)}]
\draw[thick] {(0.3,\he-0.35)--(\wi-0.3,\he-0.35)};

\draw[black,densely dashed] (\wi/3,\he) .. controls (\wi/3,\he/3*2-0.1) .. (0.27*\wi,\he/5*2-0.1);
\draw[red,densely dashed] (\wi/3,0) .. controls (\wi/3,\he/3-0.1) .. (0.27*\wi,\he/5*3-0.1);

\draw[thick] \bisection;
\draw[thick] \surface;
\end{scope}

\begin{scope}[shift={($(c)$)}]
\draw[thick] {(0.3,\he-0.35) .. controls (\wi/2,\he/2+0.1) .. (\wi-0.3,\he-0.35)};
\node[circle,draw=white,fill=white,inner sep=0pt,minimum size=5pt] at (\wi/3*2,\he/3*2) {};

\draw[red,densely dashed] (\wi/3,0)--(\wi/3,\he);
\draw[black,densely dashed] (\wi/5*4,0)--(\wi/5*4,\he);

\draw[thick] \bisection;
\draw[thick] \surface;

\node at (\wi*0.7,\he*0.35) {\small $x_2$};
\node at (\wi/3+0.1,\he*0.67+0.3) {\small $\iota(x_1)$};
\node[circle,fill=black,inner sep=0pt,minimum size=4pt] at (\wi/3,\he*0.68) {};
\node[circle,fill=green,inner sep=0pt,minimum size=4pt] at (\wi/5*4,\he*0.29) {};
\end{scope}

\begin{scope}[shift={($(d)$)}]
\draw[thick] {(0.3,\he-0.35) .. controls (\wi/2,\he/2+0.1) .. (\wi-0.3,\he-0.35)};
\node[circle,draw=white,fill=white,inner sep=0pt,minimum size=5pt] at (\wi/3*2,\he/3*2) {};

\draw[black,densely dashed] (\wi/5*4,\he) .. controls (\wi/5*4,\he/3*2-0.1) .. (0.75*\wi,\he/5*2-0.1);
\draw[green,densely dashed] (\wi/5*4,0) .. controls (\wi/5*4,\he/3-0.1) .. (0.75*\wi,\he/5*3-0.1);

\draw[thick] \bisection;
\draw[thick] \surface;

\node at (\wi/3+0.1,\he*0.67+0.3) {\small $\iota(x_1)$};
\node[circle,fill=black,inner sep=0pt,minimum size=4pt] at (\wi/3,\he*0.68) {};
\end{scope}

\begin{scope}[shift={($(e)$)}]
\draw[green,densely dashed] (\wi/3*2,0)--(\wi/3*2,\he);

\draw[thick] {(0.3,\he-0.35) .. controls (\wi/2,\he/2+0.1) .. (\wi-0.3,\he-0.35)};

\draw[thick] \bisection;
\draw[thick] \surface;

\node at (\wi*0.1,\he*0.7) {\small $D$};
\node at (\wi*0.1,\he*0.2) {\small $B$};
\node at (\wi/3+0.1,\he*0.67+0.3) {\small $\iota(x_1)$};
\node at (\wi/3*2-0.1,\he*0.67+0.3) {\small $\iota(x_2)$};
\node[circle,fill=black,inner sep=0pt,minimum size=4pt] at (\wi/3,\he*0.67) {};
\node[circle,fill=black,inner sep=0pt,minimum size=4pt] at (\wi/3*2,\he*0.67) {};
\end{scope}
\end{tikzpicture}

\vspace{-10pt}
\caption{Elementary transformation at $2n$ points for $n=1$}
\end{figure}
\end{center}

\vspace{-2em}
Let $D$ be a section on $X$ given by the strict transform of $C_0\subseteq Y$.
As we can observe from the diagram of the case $n=1$, we have
\[
\Oo_X(D)|_B=\Oo_B(\iota(x_1)+\iota(x_2)+\cdots+\iota(x_{2n}))
\]
where $\iota:B\to B$ is the involution induced by $\pi$.
Let $P_i=\Pi(x_i)$ for $i=1,2,\,\ldots,\,2n$ and $L=\Oo_C(\bb)$ be a line bundle on $C$ such that $L^2=\Oo_C(P_1+P_2+\cdots+P_{2n})$.
Then
\[
\Nm_{B/C}(\Oo_B(\iota(x_1)+\iota(x_2)+\cdots+\iota(x_{2n}))\otimes \pi^*L^{-1})=\Oo_C,
\]
and hence $E:=\pi_*(\Oo_B(\iota(x_1)+\iota(x_2)+\cdots+\iota(x_{2n}))\otimes \pi^*L^{-1})=\pi_*\Oo_B(D-\bb f)$ is an orthogonal bundle with values in $\Oo_C$ whose rank is $2$ and determinant is $M$.
Since $B$ is effective, there exists the exact sequence
\[
0
\to \Oo_X(D-\bb f-B)
\to \Oo_X(D-\bb f)
\to \Oo_X(D-\bb f)|_B
\to 0
\]
on $X$, and by pushing forward the sequence to $C$, we get $E= \Pi_*\Oo_X(D-\bb f)= \Pi_*\Oo_X(D)\otimes L^{-1}$.
Applying Proposition \ref{subbundle_after_elementary_transformation}, we have $\elm_{x_1,x_2,\ldots,x_{2n}} (\Oo_C\oplus M)=\Pi_*\Oo_X(D)\otimes \Oo_C(-P_1-P_2-\cdots -P_{2n})$ because none of $x_i$ is contained in $C_0$ due to $C_0.B=0$.
Therefore,
\[
E
=\left(\Pi_*\Oo_X(D)\otimes L^{-2}\right)\otimes L
=\elm_{x_1,x_2,\ldots,x_{2n}} (\Oo_C\oplus M)\otimes L.
\]
Thanks to Proposition \ref{img}, this argument further asserts that $E$ is semi-stable and stable in general.\linebreak
The next theorem shows that this process generates the orthogonal bundles.

\begin{Thm}\label{generation}
Let $M\in J_2(C)\backslash\{\Oo_C\}$ and $E$ be a vector bundle on $C$ of rank $2$ and determinant $M$.\linebreak
If $E$ is orthogonal with values in $\Oo_C$,
then there exist points $x_1,\,x_2\,\ldots,\,x_{2n}$ of a bisection $B'$ on the ruled surface $Y=\PP_C(\Oo_C\oplus M)$ such that $E=\elm_{x_1,\,x_2,\,\ldots, x_{2n}}(\Oo_C\oplus M)\otimes L$ for some $L\in J^n(C)$.
\end{Thm}

\begin{proof}
Let $B'\subseteq Y$ be as before and $\pi: B\to C$ be the unramified double covering corresponding to $M$\linebreak with involution $\iota:B\to B$.
Then $E=\pi_*R$ for some $R\in \Pr(B/C)$.
Fix a line bundle $L$ on $C$ of degree $g$.\linebreak
Since $2g\geq 2g-1=g(B)=\dim J^{2g}(B)$, there exist points $x_1,\,x_2,\,\ldots,\,x_{2g}\in B\cong B'$ such that
\[
\Oo_B(\iota(x_1)+\iota(x_2)+\cdots+\iota(x_{2g}))=R\otimes \pi^*L
\ \text{in $J^{2g}(B)$}
\]
and $P_i=\pi(x_i)$ satisfies $L^2=\Nm_{B/C}(R\otimes\pi^*L)=\Oo_C(P_1+P_2+\cdots+P_{2g})$ because $\Nm_{B/C}(R)=\Oo_C$.
If there are two points of the form $x$ and $\iota(x)$ in $x_i$, say $x_{2g}=\iota(x_{2g-1})$, then, by taking subtraction as
\[
\Oo_B(\iota(x_1)+\iota(x_2)+\cdots+\iota(x_{2g-2}))=R\otimes \pi^*(L(-P_g))
\ \text{in $J^{2g-2}(B)$},
\]
we may assume that there is no pair of points in $x_i$ of the form $x$ and $\iota(x)$.
Thus we can apply the previous argument to have $E\cong \elm_{x_1,x_2,\ldots,x_{2n}}(\Oo_C\oplus M)\otimes L$. 
\end{proof}

For the completeness of the exposition, we leave the following remark which states that the orthogonal bundles of the form $A^{-1}\oplus A$ are also generated by elementary transformations from $\Oo_C\oplus \Oo_C$.

\begin{Rmk}
Let $Y$ be the ruled surface $\Lambda:\PP_C(\Oo_C\oplus\Oo_C)\to C$.
We can choose two distinct sections\linebreak $C_0$ and $C_\infty$ on $Y$ corresponding to different inclusions $\Oo_C\to\Oo_C\oplus\Oo_C$.
Then they have no intersection.
For $x_1,\,\ldots,\,x_n\in C_0$ and $x_{n+1},\,\ldots,\,x_{2n}\in C_\infty$ with $P_i=\Lambda(x_i)$, we can define $\elm_{x_1,x_2,\ldots,x_{2n}} F$ whenever $\{P_1,\,\ldots,\,P_n\}\cap \{P_{n+1},\,\ldots,\,P_{2n}\}=\emptyset$.
By Proposition \ref{subbundle_after_elementary_transformation}, we have two distinct injections
\[
\Oo_C(-P_{n+1}-\cdots-P_{2n})\to \elm_{x_1,x_2,\ldots,x_{2n}}(\Oo_C\oplus\Oo_C)
\ \ \text{and}\ \ 
\Oo_C(-P_1-\cdots-P_n)\to \elm_{x_1,x_2,\ldots,x_{2n}}(\Oo_C\oplus \Oo_C).
\]
Since they destabilize $\elm_{x_1,x_2,\ldots,x_{2n}}(\Oo_C\oplus \Oo_C)$, we have
\[
\elm_{x_1,x_2,\ldots,x_{2n}}(\Oo_C\oplus\Oo_C)\otimes L=L(-P_1-\cdots-P_n)\oplus L(-P_{i+1}-\cdots-P_{2n}),
\]
for any $L\in\Pic(C)$.
As the choices of $x_i$ are arbitrary, we can generate $E=A^{-1}\oplus A$ for all $A\in J^0(C)$ in this way.
Indeed, after fixing $L\in J^g(C)$, we can find points $P_1,\,P_2\,\ldots,\,P_{2g}\in C$ which satisfy $A=L(-P_1-\cdots-P_g)$ and $A^{-1}=L(-P_{g+1}-\cdots-P_{2g})$.
If $\{P_1,\,\ldots,\,P_g\}\cap \{P_{g+1},\,\ldots,\,P_{2g}\}\neq\emptyset$, say $P_g=P_{2g}$, then we can reduce to the case $A=L'(-P_1-\cdots -P_{g-1})$ and $A^{-1}=L'(-P_{g+1}-\cdots -P_{2g-1})$ by substituting $L'=L(-P_g)$ for $L$.
Thus we can obtain that $\elm_{x_1,x_2,\ldots,x_{2n}}(\Oo_C\oplus\Oo_C)\otimes L=A^{-1}\oplus A$ for some $L\in J^n(C)$ after the substitutions.
\end{Rmk}

\vspace{1em}
\section{Semi-Stable Vector Bundles Whose Symmetric Cube is Not Stable}

\subsection{Destabilized by Rank 1}

Let $E$ be a stable vector bundle on $C$ of rank $2$.
If $S^3E$ is not stable, then it is destabilized by a subbundle of rank $1$ or $2$ by Remark \ref{first}.
We first study the case of rank $1$.

\begin{Prop}\label{tosmall}
Let $k\geq 2$.
For a vector bundle $E$ of rank $2$ and any line bundle $L$ on $C$, there exist the following exact sequences.
\begin{gather*}
0
\to \Hom(S^{k+1}E, L\otimes \det E)
\to \Hom(S^k E, E\otimes L)
\to \Hom(S^{k-1}E, L)\\
0
\to \Hom(L^{-1}\otimes \det E, S^{k+1}E)
\to \Hom(E\otimes L^{-1}, S^k E)
\to \Hom(L^{-1}, S^{k-1}E)
\end{gather*}
In particular, when $E$ is stable and has trivial determinant, for a line bundle $L$ of degree $0$,
\begin{itemize}
\item if $S^{k+1} E$ is destabilized by $S^{k+1}E\to L$, then $S^kE$ is destabilized by $S^k E\to E\otimes L$,
\item if $S^{k+1} E$ is destabilized by $L^{-1}\to S^{k+1}E$, then $S^k E$ is destabilized by $E\otimes L^{-1}\to S^k E$,
\end{itemize}
and the converses hold if $S^{k-2} E$ is stable.
\end{Prop}

\begin{proof}
Let $X=\PP_C(E)$ be the ruled surface $\Pi:\PP_C(E)\to C$ and $C_1$ be a unisecant divisor on $X$ satisfying $\Pi_*\Oo_X(C_1)=E$.
Since the natural morphism $\Pi^* E\to\Oo_C(C_1)$ is surjective, we have the following exact sequence on $X$.
\vspace{-0.5em}
\[
0\to \Oo_X(-C_1)\otimes \Pi^*\det E\to \Pi^*E\to \Oo_X(C_1)\to 0
\]
By pushing forward the sequence after twisting by $\Oo_X(kC_1)$ for $k\geq 0$, we obtain the following exact sequence on $C$.
\vspace{-0.5em}
\[\label{basic}\tag{3.1}
0\to S^{k-1}E\otimes \det E\to S^kE\otimes E\to S^{k+1}E\to 0
\]
Then, taking the dual of the sequence and twisting by $L\otimes\det E$, we have the following exact sequences.
\begin{gather*}
0
\to (S^{k+1} E)^\vee
\to (S^k E)^\vee\otimes E^\vee 
\to (S^{k-1}E)^\vee\otimes (\det E)^{-1}
\to 0\\
0
\to (S^{k+1}E)^\vee\otimes(L\otimes\det E)
\to (S^kE)^\vee\otimes (E\otimes L)
\to (S^{k-1} E)^\vee\otimes L
\to 0
\end{gather*}
By taking global sections, we get the first exact sequence of the statement.
Next, applying $E^\vee\cong E\otimes (\det E)^{-1}$ to the last sequence and twisting by $(\det E)^{k-1}$, we have the following exact sequences.
\begin{gather*}
0
\to S^{k+1}E\otimes (\det E)^{-k}\otimes L
\to S^kE\otimes (\det E)^{-(k-1)} \otimes E^\vee \otimes L
\to S^{k-1} E\otimes (\det E)^{-(k-1)}\otimes L
\to 0\\
0
\to (L^{-1}\otimes \det E)^\vee\otimes S^{k+1}E
\to (E\otimes L^{-1})^\vee\otimes S^kE
\to (L^{-1})^\vee\otimes S^{k-1}E
\to 0
\end{gather*}
By taking global sections, we obtain the second exact sequence of the statement.
\end{proof}

This fact means that if $S^{k+1}E$ is destabilized by a line bundle, then it has implications for the stability of $S^k E$.
Later, when $E$ has degree $0$, using an exact sequence (see Lemma \ref{generalization}) which generalizes \eqref{basic}, we can deduce that if $S^{k+1}E$ is destabilized by a line subbundle, then, for all $\frac{k+1}{2}<m\leq k$, $S^m E$ is not stable.\linebreak
In the opposite direction, we can show the following.

\begin{Prop}\label{tolarge}
Let $E$ be a stable vector bundle on $C$ of rank $2$ and degree $0$.
If $S^k E$ is destabilized by\linebreak a line subbundle $L^{-1}\to S^k E$, then $S^l E$ is destabilized by a subbundle $S^{l-k}E\otimes L^{-1}\to S^l E$ for all $l\geq k$.
\end{Prop}

\begin{proof}
If there is a line subbundle $L^{-1}\to S^k E$ of degree $0$, then there exists a $k$-section $D\sim kC_1+\bb f$ on $X$ for $L=\Oo_C(\bb)$ with projection $\pi=\Pi_D|:D\to C$.
Thus it gives the exact sequence
\[
0\to\Oo_X((l-k)C_1-\bb f)\to \Oo_X(lC_1)\to \Oo_D(lC_1)\to 0
\]
on $X$, and by pushing forward the sequence, we have the following exact sequence on $C$.
\[
0\to S^{l-k}E\otimes L^{-1}\to S^lE\to \pi_*\Oo_C(lC_1)|_D\to 0
\]
Because $\deg(S^{l-k}E\otimes L^{-1})=\deg(S^lE)=0$, $S^{l-k}E\otimes L^{-1}$ destabilizes $S^l E$.
\end{proof}

If $S^3 E$ is destabilized by a line subbundle, then $S^2 E$ is not stable by Proposition \ref{tosmall}, so $S^2 E$ has\linebreak a destabilizing line subbundle (see Remark \ref{first}).
Conversely, if $S^2 E$ is destabilized by a line subbundle, then $S^3 E$ is not stable by Proposition \ref{tolarge}, but we do not know whether $S^3 E$ has a destabilizing subbundle.
The following tells when $S^3 E$ is destabilized by a line subbundle.

\begin{Thm}\label{classification_rank_one}
Let $E$ be a stable vector bundle on $C$ of rank $2$ with trivial determinant.
If $S^3 E$ is destabilized by a line subbundle $L^{-1}\to S^3 E$, then $L\in J_4(C)$ and $E=(\pi_*R)\otimes L$ for some $R\in\Pr(B/C)$ with $R\in J_6(B)$ where $\pi:B\to C$ is the unramified double covering corresponding to $L^2\in J_2(C)$.
\end{Thm}

\newgeometry{top=80pt,bottom=80pt,left=80pt,right=80pt,footskip=30pt}
\begin{proof}
Assume that $S^3 E$ is destabilized by a line subbundle $L^{-1}\to S^3 E$ of $\deg L=0$.
Then $S^2 E$ is destabilized by the subbundle $E\otimes L^{-1}\to S^2 E$ by Proposition \ref{tosmall}.
Completing the quotient by comparing the determinants, we have the following exact sequence.
\[
0\to E\otimes L^{-1}\to S^2 E\to L^2\to 0
\]
We can observe from Remark \ref{twotor} with the surjection $S^2 E\to L^2$ that $L^2\in J_2(C)$.
Notice that the dual of the surjection, $L^{-2}\to (S^2 E)^\vee\cong S^2 E$, yields a bisection $B\sim 2C_1+2\bb f$ on $X=\PP_C(E)$ for $L=\Oo_C(\bb)$. Let $\pi=\Pi|_B:B\to C$ be the induced unramified double covering.
By pushing forward the exact sequence
\[
0\to \Oo_X(C_1-2\bb f)\to \Oo_X(3C_1)\to \Oo_B(3C_1)\to 0
\]
on $X$, we obtain the following exact sequence on $C$.
\[
0\to E\otimes L^{-2}\to S^3 E\to \pi_*\Oo_B(3C_1)\to 0
\]
Since there cannot exist a nonzero morphism $L^{-1}\to E\otimes L^{-2}$ as $E$ is stable, the morphism $L^{-1}\to S^3 E$ induces a nonzero morphism $L^{-1}\to \pi_*\Oo_B(3C_1)$, and it implies that $\pi_*\Oo_B(3C_1)$ is not stable.

By pushing forward the exact sequence
\[
0\to \Oo_X(-C_1-3\bb f)\to\Oo_X(C_1-\bb f)\to\Oo_B(C_1-\bb f)\to 0
\]
on $X$ to $C$, we get $E\otimes L^{-1}=\Pi_*\Oo_X(C_1-\bb f)=\pi_*\Oo_B(C_1-\bb f)$.
Note that $E\otimes L^{-1}$ is an orthogonal bundle with values in $\Oo_C$ as there is a surjection $S^2(E\otimes L^{-1})\to \Oo_C$.
Then $\Oo_B(C_1-\bb f)\in\Pr(B/C)$ follows from Mumford's classification.

Because the stability of $\pi_*\Oo_B(3C_1-3\bb f)=(\pi_*\Oo_B(3C_1))\otimes \Oo_C(-3\bb)$ and $\pi_*\Oo_B(3C_1)$ are equivalent, and $\Oo_B(3C_1-3\bb f)\in\Pr(B/C)$, we can deduce that $\pi_*\Oo_B(3C_1-3\bb f)$ splits by Proposition \ref{img}.
Therefore, $\Oo_B(3C_1-3\bb f)$ is $2$-torsion, and hence $\Oo_B(C_1-\bb f)$ is $6$-torsion.
\end{proof}

We can verify the converse of the theorem as in the next remark.
That is, if $R\in\Pr(B/C)\cap J_6(B)$ and $E=\pi_*R$, then $S^3 E$ is destabilized by a line subbundle, but we need to exclude $2$-torsion $R$ for $E$ to be stable.
For all nontrivial $M\in J_2(C)$ and a line bundle $A$ with $A^2=M$, recall from Section 2 that the images of $\Phi_A:\Pr(B/C)\to\SU_C(2,\Oo_C)$ are the locus of $E$ with strictly semi-stable $S^2 E$.
Then, $E=\Phi_A(R)$ for some $R\in\Pr(B/C)\cap(J_6(B)\backslash J_2(B))$ if and only if $E$ is stable and $S^3 E$ is destabilized by a line subbundle. 
Since the number of choices of $M\in J_2(C)$ and $R\in J_6(B)$ are finite, we can conclude that there are only finitely many such $E$ in $\SU_C(2,\Oo_C)$.

\begin{Rmk}
Let $\pi: B\to C$ be an unramified double covering corresponding to $M\in J_2(C)\backslash\{\Oo_C\}$.\linebreak
If $k\geq3$, then we can find in the same manner that $E=\pi_*R$ is stable and $S^k E$ is destabilized by a line subbundle for $R\in\Pr(B/C)\cap (J_{2k}(B)\backslash J_2(B))$.
Let $X=\PP_C(E)$ be the ruled surface $\Pi:\PP_C(E)\to C$ and $C_1$ be a unisecant divisor on $X$ such that $\Pi_*\Oo_X(C_1)=E$.
As $S^2 E$ is destabilized by $\Oo_C$, $B$ realizes as a bisection $B\sim 2C_1$ on $X$, and then $E=\pi_*\Oo_B(C_1)$.
By pushing forward the exact sequence
\[
0\to\Oo_X((k-2)C_1)\to \Oo_X(kC_1)\to \Oo_B(kC_1)\to 0
\]
on $X$, we obtain the following exact sequence on $C$.
\[
0\to S^{k-2}E\to S^k E\to \pi_*\Oo_B(kC_1)\to 0
\]
Notice that $E$ is orthogonal with values in $\Oo_C$.
Thus $\Oo_B(C_1)\in\Pr(B/C)$, and so $\Oo_B(kC_1)\in\Pr(B/C)$.
Hence $\Oo_B(kC_1)\in J_2(B)$ is equivalent to saying that $\pi_*\Oo_B(kC_1)$ splits into the direct sum of line bundles by Proposition \ref{img}.
Therefore, $S^k E$ is destabilized by a line 
(sub)bundle when $\Oo_B(C_1)$ is $2k$-torsion.
\end{Rmk}

\newgeometry{top=65pt,bottom=80pt,left=80pt,right=80pt,footskip=30pt}
\subsection{Destabilized by Rank 2 but not by Rank 1}

In Proposition \ref{tolarge}, we observe that if $S^2 E$ is strictly semi-stable, then $S^3 E$ is destabilized by a subbundle of rank $2$.
The converse does not hold as there is an example where $S^2 E$ is stable but $S^4 E$ is destabilized by a line subbundle.
By Proposition \ref{tosmall},\linebreak if $S^4 E$ is destabilized by a line subbundle, then $S^3 E$ is destabilized by a subbundle of rank $2$.

\begin{Ex}\label{S3}
Let $\eta:C'\to C$ be an unramified cyclic triple covering corresponding to $L\in J_3(C)\backslash\{\Oo_C\}$.
Then we have $\eta_*\Oo_{C'}=\Oo_C\oplus L^{-1}\oplus L^{-2}$.
For $M\in J_2(C')$, $V=\eta_*M$ becomes an orthogonal bundle on $C$\linebreak of rank $3$.
Indeed, the surjection
\[
\eta^*S^2 V
=S^2(\eta^* V)
=S^2(\eta^*(\eta_*M))
\to M^2
=\Oo_{C'}
\]
induces a nonzero morphism $S^2 V\to \eta_*\Oo_{C'}=\Oo_C\oplus L^{-1}\oplus L^{-2}$, and it implies that $V$ is orthogonal because
one of $S^2 V\to L^{-k}$ is nonzero which is necessarily surjective due to the semi-stability of $S^2 V$.
In fact, there are surjections $S^2 V\to L^{-k}$ for all $k$ (mod $3$);
if there is a surjection $S^2V\to L^{-k}$ for one $k$,\linebreak then the isomorphism $V\otimes L=\eta_*M\otimes L\cong \eta_*(M\otimes \eta^*L)\cong \eta_*M=V$ induces another surjection $S^2 V\to L^{-k+1}$ after twisting the surjection $S^2V\otimes L^2\cong S^2 (V\otimes L)\cong S^2V\to L^{-k}$ by $L$.

According to Mumford's classification \cite{Mumford71} of orthogonal bundles of rank $3$, $V\cong A^{-1}\otimes S^2 E$ for some vector bundle $E$ of rank $2$ and line bundle $A$ satisfying $A^2=(\det E)^2$.
We will first check that $S^2 E$ is stable if and only if $M\not\in \eta^*J^0(C)$, and then show that $S^4 E$ has a destabilizing quotient line bundle.

If $\eta_*M$ is destabilized by a line subbundle $R^{-1}\to \eta_*M$ of degree $0$, then there is a nonzero morphism $\eta^*R^{-1}\to M$.
Since both $M$ and $\eta^*R^{-1}$ have degree $0$, it is possible if and only if $M=\eta^*R^{-1}\in \eta^*J^0(C)$.

Next, to show that $S^4 E$ is destabilized by the quotient line bundles $S^4 E\to A^2\otimes L^{\pm 1}$, we will use the following exact sequence obtained by completing the kernel of the natural surjection $S^2(S^2 E)\to S^4 E$ after comparing the determinants.
\vspace{-0.3em}
\[
0\to (\det E)^2\to S^2 (S^2 E)\to S^4 E\to 0
\]
Twisting the sequence by $A^{-2}\cong (\det E)^2$, we have the following exact sequence.
\[
0\to\Oo_C\to S^2(A^{-1}\otimes S^2 E)\to S^4 E\otimes A^{-2}\to 0
\vspace{-0.2em}
\]
Because there exist the surjections $S^2 V\to L^{\pm 1}$ and $L\neq \Oo_C$,
the above sequence with $V\cong A^{-1}\otimes S^2 E$ induces surjections $S^4 E\otimes A^{-2}\to L^{\pm 1}$.
That is, $S^4 E$ has destabilizing quotient bundles $A^2\otimes L^{\pm 1}$.
\end{Ex}

Note that the number of $E$ given by the above construction is finite since there are only finitely many $m$-torsion line bundles for each $m$ and vector bundles $E$ with $\det E=\det F$ which satisfy $S^2 E=S^2 F$ for a fixed vector bundle $F$ of rank $2$ with even degree when $S^2 F$ is stable.
As the first fact is well-known, we give a proof for the latter fact.

\begin{Prop}\label{fundamental_theorem_of_square_tensor}
Let $E$, $F$ be vector bundles on $C$ of rank $2$ with the same determinant of even degree.
If $S^2 E$, $S^2 F$ are stable and $S^2E\cong S^2 F$, then $E\cong F\otimes M$ for some $2$-torsion line bundle $M$.
\end{Prop}

\begin{proof}
We may give a proof under the assumption that $\det E=\det F=\Oo_C$.
Notice that $E$, $F$ are stable because $S^2 E$, $S^2 F$ are stable.
Due to $E\otimes E=\Oo_C\oplus S^2 E$, $E^\vee\cong E$, and the same facts for $F$, we have
\begin{align*}
\dim\Hom(E\otimes F,E\otimes F)
&=h^0(E^\vee\otimes F^\vee\otimes E\otimes F)
=\dim\Hom(E\otimes E,F\otimes F)\\
&\geq \dim\Hom(\Oo_C,\Oo_C)
+\dim\Hom(S^2 E,S^2 F)
=2.
\end{align*}
Thus $E\otimes F$ is not simple, so not stable, hence there exists a subbundle $V\to E\otimes F$ of rank $r\leq 2$ and degree $0$ by Remark \ref{first}.
If $r=1$, then it induces an isomorphism $E\cong E^\vee\cong F\otimes V^{-1}$ as $E$ and $F$ are stable.
After comparing the determinants, we can check that $M^2=\det (F\otimes V^{-1})=\det E=\Oo_C$ for $M=V^{-1}$ which gives an isomorphism $E\cong F\otimes M$.

\begin{center}\vspace{25pt}\end{center}
Otherwise, if there is no destabilizing line subbundle of $E\otimes F$ so that $r=2$, then we can derive\linebreak a contradiction as follows.
Suppose that there is an exact sequence
\[
0\to V\to E\otimes F\to W\to 0
\]
for some vector bundles $V$ and $W$ of rank $2$ and degree $0$.
If $W$ is not stable so there exists a destabilizing quotient line bundle $W\to L$, then it gives a destabilizing line subbundle $L^{-1}\to W^\vee\to (E\otimes F)^\vee\cong E\otimes F$.
If $V$ is not stable, then we have the same contradiction.
So we assume that both $V$ and $W$ are stable.
Then, using the filtration
\[
S^2(E\otimes F)=G_0\supseteq G_1\supseteq G_2\supseteq G_3=0
\]
satisfying $G_0/G_1=S^2 W$, $G_1/G_2=V\otimes W$, and $G_2/G_3=S^2 V$, we obtain that
\[
h^0(S^2 (E\otimes F))\leq h^0(S^2V)+h^0(S^2 W)+h^0(V\otimes W).
\]
On the other hand, by the isomorphism
\[
S^2(E\otimes F)
\cong(\det E\otimes \det F)\oplus (S^2E\otimes S^2 F)
=\Oo_C\oplus (S^2 E\otimes S^2 F),
\]
we have $h^0(S^2(E\otimes F))=h^0(\Oo_C)+\dim \Hom((S^2 E)^\vee, S^2 F)=1+\dim\Hom(S^2 E, S^2 F)=2$, and therefore,
\[
h^0(S^2 V)+h^0(S^2 W)+h^0(V\otimes W)\geq 2.
\]

We will first show that $h^0(S^2 V)=h^0(S^2 W)=1$.
That is, $V$ and $W$ are orthogonal with values in $\Oo_C$.\linebreak
Since $V\otimes V=\det V\oplus S^2 V$ and $V$ is stable, $h^0(S^2 V)\leq h^0(V\otimes V)=\dim\Hom(V^\vee,V)\leq 1$.
For the same reason, $h^0(S^2 W)\leq 1$.
Thus it suffices to treat the case $h^0(V\otimes W)\neq 0$.
Because both $V$ and $W$ are stable, $h^0(V\otimes W)=\dim\Hom(V^\vee,W)\leq 1$, and the equality holds if and only if $W\cong V^\vee$.
Hence if $h^0(V\otimes W)\neq 0$, then $h^0(V\otimes W)=1$, so we have either $h^0(S^2 V)=1$ or $h^0(S^2 W)= 1$ from $h^0(S^2 V)+h^0(S^2 W)\geq 1$.
Without loss of generality, we may assume that $h^0(S^2 V)= 1$, equivalently, there is a symmetric isomorphism $V^\vee\cong V$.
Then $W\cong V^\vee\cong V$ is true due to $h^0(V\otimes W)\neq 0$.
Thus we get $S^2 W\cong S^2 V$ as well, and so $h^0(S^2 W)=h^0(S^2 V)=1$.

Now suppose that $V$ and $W$ are orthogonal with values in $\Oo_C$.
By Remark \ref{sss}, we know that $\det V$ and $\det W$ are nontrivial $2$-torsion.
From $\det V\otimes \det W=\det(E\otimes F)=\Oo_C$, we also have $\det V=\det W$, and denote them by $N=\det V=\det W\in J_2(C)$.
Then, from $V\otimes N\cong V^\vee\cong V$ and $W\otimes N\cong W^\vee\cong W$, we obtain that $(E\otimes F)\otimes N\cong E\otimes F$ since $E\otimes F$ is an extension of $W$ by $V$.
Thus
\begin{align*}
(E\otimes F)\otimes (E\otimes F)
&\cong(E\otimes F)\otimes ((E\otimes F)\otimes N)
=((E\otimes E)\otimes (F\otimes F))\otimes N\\
&=((\Oo_C\oplus S^2 E)\otimes (\Oo_C\oplus  S^2 F))\otimes N
=(\Oo_C\oplus S^2 E\oplus S^2 F\oplus (S^2 E\otimes S^2 F))\otimes N\\
&=N\oplus (S^2 E\otimes N)\oplus (S^2 F\otimes N)\oplus (S^2 E\otimes S^2 F\otimes N)
\end{align*}
implies that $h^0(N)+h^0(S^2 E\otimes N)+h^0(S^2 F\otimes N)+h^0(S^2 E\otimes S^2 F\otimes N)=\dim\Hom(E\otimes F,E\otimes F)=2$.

Because $S^2 E\cong S^2 F$ are stable, $h^0(S^2 E\otimes N)=h^0(S^2 F\otimes N)=0$.
Moreover, due to their stability,
\[
h^0(S^2 E\otimes S^2 F\otimes N)
=\dim\Hom((S^2 E)^\vee,S^2 F\otimes N)
=\dim\Hom(S^2 E,S^2 F\otimes N)
\leq 1.
\]
Thus it follows that $h^0(N)=1$ which is equivalent to saying that $\det V=N=\Oo_C$.
Then it contradicts the inequality $h^0(\det V)+h^0(S^2 V)=h^0(V\otimes V)=\dim\Hom(V^\vee,V)\leq 1$ satisfied by the stability of $V$.\linebreak
Therefore, there must be a line subbundle $M^{-1}\to E\otimes F$ yielding an isomorphism $E\cong E^\vee\cong F\otimes M$.
\end{proof}

The remaining part of this subsection is devoted to showing that the vector bundles constructed in Example \ref{S3} are the only vector bundles $E$ where $S^2 E$ is stable but $S^3 E$ is not stable.
\restoregeometry

\begin{Lem}\label{integral}
Let $E$ be a vector bundle on $C$ of rank $2$ and degree $0$.
Assume that $S^k E$ is destabilized by a line subbundle so there exists a $k$-section $D$ of zero self-intersection on the ruled surface $X=\PP_C(E)$.
Then $D$ is irreducible and reduced if $S^m E$ is not destabilized by a line subbundle for any $m\leq \frac{k}{2}$, or equivalently, if there is no $m$-section of zero self-intersection on $X$ for any $m\leq \frac{k}{2}$ (see Remark \ref{interesting}).
\end{Lem}

\begin{proof}
We fix a unisecant divisor $C_1$ on $X$ satisfying ${C_1}^2=0$.
Then necessarily $D\equiv kC_1$.
If $D$ is neither irreducible nor reduced, then $D=D_1\cup D_2$ for some effective divisors $D_1\equiv mC_1+bf$ and $D_2\equiv nC_1+ df$ with $m+n=k$ for $m,\,n>0$ and $b+d=0$ for $b,\,d\geq 0$.
So we have $b=d=0$ and both $D_1$ and $D_2$ have zero self-intersection on $X$.
Thus it yields a contradiction as one of $m\leq \frac{k}{2}$ or $n \leq \frac{k}{2}$ must hold.
\end{proof}

\begin{Lem}\label{strong_and_concise}
Let $E$ be a vector bundle on $C$ of rank $2$ and degree $0$ and $X=\PP_C(E)$ be the ruled surface $\Pi: \PP_C(E)\to C$ with a unisecant divisor $C_1$ on $X$ satisfying $\Pi_*\Oo_X(C_1)\cong E$.
Assume that $k\geq 2$ and $S^m E$ is not destabilized by a line subbundle for any $m\leq k$.
If there exists an injection $V\to S^k E$ for some stable vector bundle $V$ on $C$ of rank $2$ and degree $0$, then it induces a surjection $\Pi^*V\to \Oo_X(kC_1)$~on~$X$.
\end{Lem}

\begin{proof}
There is a nonzero morphism $\morph:\Pi^*V\to\Oo_X(kC_1)$ from the adjoint correspondence
\[
\Hom(\Pi^*V,\Oo_X(kC_1))
\cong\Hom(V,\Pi_*\Oo_X(kC_1))
=\Hom(V,S^kE).
\]
Since $\im \morph$ is a torsion-free sheaf on $X$ of rank $1$, we can write $\im \morph=\Oo_X(mC_1+\dd f)\otimes \Ii_Z$ for some integer $m$, $\dd\in\Pic(C)$, and $0$-dimensional subscheme $Z$ of $X$, which may be empty.
As the injection $\im \gamma\to \Oo_X(kC_1)$ induces a nonzero morphism $\Oo_X(mC_1+\dd f)\to\Oo_X(kC_1)$ between the reflexive hulls, $m\leq k$ and $\deg\dd\leq 0$ must hold.

Let $P\in C$ be arbitrary and $Pf=\Pi^{-1}(P)$ denote the $\PP^1$-fiber of $\Pi$ over $P$.
Since there is a trivialization $V|_U\cong \Oo_U\oplus \Oo_U$ in a neighborhood $U\subseteq C$ of $P\in C$, we get $\Pi^* V|_{Pf}\cong \Oo_{\PP^1}\oplus \Oo_{\PP^1}$.
Then, by restricting the surjection $\Pi^*V\to \im \morph$ onto $Pf$, we obtain a surjection $\Oo_{\PP^1}\oplus\Oo_{\PP^1}\to \Oo_{\PP^1}(m)$.
Since $\Oo_{\PP^1}(m)$ cannot be generated by global sections when $m<0$, we get $m\geq 0$.

On the other hand, by restricting the surjection $\Pi^*V\to \im \morph$ onto an $l$-section $D\sim lC_1+\cc f$ for some $\cc\in\Pic(C)$ of $\deg\cc=c$, we have a surjection $\Pi^*V|_D\to \Oo_X(mC_1+\dd f)\otimes\Ii_Z|_D$.
Because $\pi=\Pi|_D: D\to C$ is finite, $\Pi^*V|_D=\pi^*V$ is a semi-stable vector bundle on $D$ of degree $0$ \cite[II:~p.~61]{Lazarsfeld04}.
Thus
\[
0\leq \deg \Oo_X(mC_1+\dd f)\otimes\Ii_Z|_D\leq (mC_1+\dd f).(lC_1+\cc f)=mc+l\deg\dd.
\]
Since the $\RR$-ray generated by $[C_1]$ lies on the boundary of the cone of ample divisors in $N^1(X)$, we can choose a smooth curve $D\equiv lC_1+cf$ with $c/l>0$ being arbitrary small by taking sufficiently large $l\gg 0$.
So $\deg \dd$ is necessarily $0$.

Let $L=\Oo_C(\bb)=\det V$ and consider the exact sequence
\[
0
\to \Oo_X(nC_1+\ee f)\otimes \Ii_W
\to \Pi^*V
\to \Oo_X(mC_1+\dd f)\otimes\Ii_Z\to 0
\]
on $X$ for some integer $n$, $\ee\in\Pic(C)$, and $0$-dimensional subscheme $W$ of $X$, which is possibly empty.
Applying Whitney's formula for coherent sheaves to the sequence, we get
\[
c_1(\Pi^*V)=(mC_1+\dd f)+(nC_1+\ee f)
\ \ \text{and}\ \ 
c_2(\Pi^*V)=(mC_1+\dd f).(nC_1+\ee f) +\len(Z)+\len(W).
\]
Because $c_1(\Pi^*V)=\bb f$, $c_2(\Pi^*V)=0$, and ${C_1}^2=0$, we obtain $n=-m$, $\ee=\bb-\dd$, and $Z=W=\emptyset$.
Then, from the inclusion $\Oo_X(mC_1+\dd f)\to \Oo_X(kC_1)$, we get a global section of $\Oo_X((k-m)C_1-\dd f)$, which gives a $(k-m)$-section of zero self-intersection on $X$ if $m<k$, so it contradicts the assumption.
Therefore, $m=k$ and hence $\dd=0$, which together imply that $\gamma$ is surjective.
\end{proof}

\newgeometry{top=80pt,bottom=80pt,left=80pt,right=80pt,footskip=30pt}
\begin{Prop}\label{last}
Let $E$ be a vector bundle on $C$ of rank $2$ with trivial determinant.
Assume that $S^2 E$ is stable and $S^3 E$ is not destabilized by a line subbundle.
If $S^3E$ is destabilized by a subbundle $V\to S^3 E$ of rank $2$, then $L=\det V$ is nontrivial $3$-torsion, and $S^3 E$ has two distinct destabilizing subbundles $E\otimes L^{\mp 1}\to S^3 E$.
\end{Prop}

\begin{proof}
Assume that there exists a nonzero morphism $V\to S^3E$ for some stable vector bundle $V$ on $C$ of rank $2$ with $\det V= L$ of $\deg L=0$.
Then, by Lemma \ref{strong_and_concise}, the induced morphism $\Pi^*V\to \Oo_X(3C_1)$ must be surjective where $X=\PP_C(E)$ is the ruled surface $\Pi:\PP_C(E)\to C$ and $C_1$ is a unisecant divisor on $X$\linebreak such that $E=\Pi_*\Oo_X(C_1)$.
Then we have the following exact sequence on $X$ by completing the kernel after comparing the determinants.
\[
0\to\Oo(-3C_1)\otimes \Pi^*L\to \Pi^*V\to \Oo_X(3C_1)\to 0
\]
By pushing forward the above sequence, we get the exact sequence
\[
0\to V\to S^3 E\to E^\vee\otimes L\to 0
\]
on $C$, and it gives that $L^3=\Oo_C$ by comparing the determinants.
Moreover, by taking the dual of the sequence, we obtain an injection $E\otimes L^{-1}\to (S^3 E)^\vee\cong S^3 E$.

Assume that $L$ is nontrivial.
Due to Lemma \ref{strong_and_concise} with respect to $E\otimes L^{-1}\to S^3 E$, we can say that there exists the following exact sequence on $X$.
\[
0\to \Oo_X(-3C_1)\otimes \Pi^*L^{-2}\to \Pi^*(E\otimes L^{-1})\to \Oo_X(3C_1)\to 0
\]
By pushing forward the exact sequence after twisting by $\Oo_X(-C_1)$, we have the exact sequence
\[\label{twisting_square}\tag{3.2}
0\to S^2 E\to S^2 E\otimes L^{-2}\to 0
\]
on $C$ due to $(S^2 E)^\vee\cong S^2 (E^\vee)\cong S^2 E$.
Thus $S^2E\cong S^2E\otimes L$.
Twisting both sides by $L$, we also obtain that $S^2E\otimes L\cong S^2E\otimes L^2$.
That is, $S^2E\otimes L^{-1}\cong S^2E\cong S^2E\otimes L$.
As shown in Example \ref{S3}, there is the following exact sequence.
\[
0\to \Oo_C\to S^2(S^2 E)\to S^4 E\to 0
\]
Notice that $H^0(S^2(S^2 E))\neq 0$.
Twisting the sequence by $L^{\pm 1}$, we get the following exact sequence.
\[
0\to L^{\pm 1}\to S^2(S^2 E)\to S^4E\otimes L^{\pm 1}\to 0
\]
Since $L\neq\Oo_C$, $H^0(S^2(S^2 E))\neq 0$ implies that $H^0(S^4 E\otimes L^{\pm 1})\neq0$.
So there are nonzero morphisms $L^{\mp 1}\to S^4 E$, which respectively induce destabilizations $E\otimes L^{\mp 1}\to S^3 E$ by Proposition \ref{tosmall}.

Now it remains to eliminate the case $L=\Oo_C$.
Suppose that there is an injection $E\to S^3 E$.
Applying Proposition \ref{tosmall}, we can see that there exist a $4$-section $D\sim 4C_1$ and the following exact sequence on $X$.
\[
0\to \Oo_X(-8C_1)\to \Oo_X(-4C_1)\to \Oo_D(-4C_1)\to 0
\]
By pushing forward the sequence to $C$, we obtain an injection $H^0(\pi_*\Oo_D(-4C_1))\to H^0(S^6 E)$ after taking global sections where $\pi:D\to C$ is the induced unramified $4$-covering.
As $S^2 E$ is assumed to be stable so that $E$ is stable as well, $D$ is irreducible and reduced by Lemma \ref{integral}.
Due to Proposition \ref{k-2},  $\Oo_D(2C_1)=\Oo_D$, and thus $H^0(S^6 E)\neq 0$ follows from $H^0(\pi_*\Oo_D(-4C_1))=H^0(\pi_*\Oo_D)=H^0(\Oo_D)\neq 0$.

Let $B\sim 6C_1$ be the $6$-section on $X$ corresponding to a nonzero global section of $S^6 E$, and we denote the induced $6$-covering by $\xi=\Pi|_B: B\to C$.
Suppose that $B$ is irreducible and reduced.
Then, from Remark \ref{interesting}, $B$ is smooth.
By pushing forward the exact sequence
\[
0\to \Oo_X(-6C_1)\to \Oo_X\to \Oo_B\to 0
\]
on $X$ to $C$, we can see that $\xi_*\Oo_B=\Oo_C\oplus S^4 E$ as there is a splitting of the natural inclusion $\Oo_C\to\xi_*\Oo_B$ \cite[I:~p.~248]{Lazarsfeld04}.
Because $H^0(S^4 E)\neq 0$, we get $h^0(\Oo_B)=h^0(\xi_*\Oo_B)\geq 2$, contradicting the hypothesis on $B$.\linebreak
Therefore, $B$ is neither irreducible nor reduced, and so by Lemma \ref{integral}, we have a contradiction to the assumption that none of $E$, $S^2 E$, and $S^3 E$ are destabilized by a line subbundle.
\end{proof}

We continue to classify vector bundles $E$ of rank $2$ with trivial determinant where $S^2 E$ is stable but there exist subbundles $E\otimes L^{\mp 1}\to S^3 E$ for some nontrivial $3$-torsion line bundle $L$.

\begin{Prop}\label{sumup}
Let $E$ be a vector bundle on $C$ of rank $2$ with trivial determinant.
Assume that $S^2 E$ is stable and $S^3 E$ has destabilizing subbundles $E\otimes L^{\mp 1}\to S^3 E$ for some $L\in J_3(C)\backslash\{\Oo_C\}$.
Then there exists an irreducible and reduced $6$-section $B$ on $X=\PP_C(E)$ and $\xi^*L=\Oo_B$ for the induced $6$-covering $\xi: B\to C$.
\end{Prop}

\begin{proof}
Using Lemma \ref{strong_and_concise} with respect to the inclusions $E\otimes L^{\mp 1}\to S^3 E$, we can further assert that there exist two exact sequences
\[
0\to \Oo_X(-3C_1)\otimes \Pi^*L^{\mp 2}\to \Pi^*(E\otimes L^{\mp 1})\to \Oo_X(3C_1)\to 0
\]
on $X$.
After twisting by $\Pi^*L^{\pm 1}=\Pi^*\Oo_C(\pm\bb)=\Oo_X(\pm\bb f)$, the sequences are in other way written by
\[
0\to \Oo_X(-3C_1\mp \bb f)\to \Pi^*E\to \Oo_X(3C_1\pm\bb f)\to 0.
\]
Since $L$ is nontrivial $3$-torsion, $\Oo_X(3C_1+\aa f)\neq \Oo_X(3C_1-\aa f)$.
So, by adjoining
\[
0\to \Oo_X(-3C_1+\bb f)\to \Pi^*E
\quad\text{and}\quad
\Pi^*E\to\Oo_X(3C_1+\bb f)\to 0,
\]
we get a nonzero morphism $\Oo_X(-3C_1+\bb f)\to \Oo_X(3C_1+\bb f)$ which yields a global section of $\Oo_X(6C_1)$.
Thus if $S^3 E$ has destabilizing subbundles $E\otimes L^{\mp 1}\to S^3 E$ for some $L\in J_3(C)\backslash\{\Oo_C\}$, then there exists a~$6$-section $B\sim 6C_1$ on $X$.
Furthermore, by Lemma \ref{integral}, $B$ is irreducible and reduced as $S^2 E$ is stable.

Let $\xi=\Pi|_B:B\to C$ be the induced $6$-covering.
By restricting the three surjections
\[
\Pi^*E\to \Oo_X(3C_1+\bb f),\ \ 
\Pi^*E\to \Oo_X(3C_1-\bb f),\ \ \text{and}\ \ 
\Pi^*E\to \Oo_X(C_1)
\]
on $X$ to $B$, we have the following three quotient line bundles on $B$ of degree $0$.
\[\label{three_quotients}\tag{3.3}
\xi^*E\to \Oo_B(3C_1+\bb f),\ \ 
\xi^*E\to \Oo_B(3C_1-\bb f),\ \ \text{and}\ \
\xi^*E\to \Oo_B(C_1)
\]

As $B\sim 6C_1$ is effective, by pushing forward the exact sequences
\[
0\to \Oo_X(-4C_1\pm\bb f)\to \Oo_X(2C_1\pm\bb f)\to \Oo_B(2C_1\pm\aa f)\to 0
\]
on $X$, we have the exact sequences
\[\label{enclosed_by_squares}\tag{3.4}
0\to S^2 E\to \xi_*\Oo_B(2C_1\pm\bb f)\to S^2 E\to 0
\]
on $C$ because $(S^2 E)^\vee\cong S^2 E$ and $S^2 E\cong S^2E\otimes L^{\pm 1}$ from \eqref{twisting_square} in the proof of Proposition \ref{last}.

Now suppose that $\xi^*L=\Oo_B(\bb f)$ is nontrivial.
Because $\xi^*E$ is a semi-stable vector bundle on $B$ of rank~$2$ and degree $0$, it has at most two distinct isomorphic types of quotient line bundles of degree $0$.
Since the first and second quotients in \eqref{three_quotients} are distinct, the third quotient must be equal to one of the others.
Thus we get either
\[
\Oo_B(2C_1+\bb f)\cong \Oo_B
\quad \text{or}\quad
\Oo_B(2C_1-\bb f)\cong \Oo_B.
\]
So, when $\Oo_B(\bb f)\neq \Oo_B$, $H^0(S^2 E)\neq 0$ follows from one of \eqref{enclosed_by_squares} together with $h^0(\xi_*\Oo_B)=h^0(\Oo_B)=1$, and it gives rise to the strict semi-stability of $S^2 E$, contradicting the assumption.
Thus $\xi^*L=\Oo_B$.
\end{proof}
\restoregeometry

\begin{Thm}\label{classification_rank_two}
Let $E$ be a vector bundle on $C$ of rank $2$ with trivial determinant.
Assume that $S^2 E$ is stable but $S^3 E$ is not destabilized by a line subbundle.
If $S^3E$ is destabilized by a subbundle of rank $2$,\linebreak then $S^2 E=\eta_*M$ for some nontrivial unramified cyclic covering $\eta: C'\to C$ of degree $3$ and $2$-torsion line bundle $M$ on $C'$ which is not contained in $\eta^*J_0(C)$.
\end{Thm}

\begin{proof}
Let $X=\PP_C(E)$ be the ruled surface $\Pi: \PP_C(E)\to C$ and $C_1$ be a unisecant divisor on $X$ such that $\Pi_*\Oo_X(C_1)=E$.
Under the same assumption, we have seen in Proposition \ref{last} and \ref{sumup} that
\begin{enumerate}
\item there exists an irreducible and reduced $6$-section $B\sim 6C_1$ on $X$ with projection $\xi=\Pi|_B: B\to C$,
\item there exists a nontrivial $3$-torsion line bundle $L=\Oo_C(\bb)$ on $C$ and $\xi^*L$ is trivial,
\item from \eqref{three_quotients}, $\xi^*E$ has quotient line bundles $\Oo_B(C_1)$ and $\Oo_B(3C_1)$ of degree $0$, and
\item from \eqref{enclosed_by_squares}, there is the following exact sequence on $C$.
\[
0\to S^2E\to \xi_*\Oo_B(2C_1)\to S^2E\to 0
\]
\end{enumerate}
In addition, from (2), $\xi$ must be factored into
\[
\xymatrix@1{
\xi:\ B\ \ar[r]^{\sigma} &
\ C'\ \ar[r]^{\eta} &
\ C 
}
\]
for some unramified cyclic triple covering $\eta$ and unramified double covering $\sigma$ \cite[Proposition 11.4.3]{Birkenhake-Lange04}.

First, notice that $\eta^*E$ is stable.
For, otherwise, its destabilizing line bundle $\eta^*E\to S$ of degree $0$ gives a morphism $\phi:C'\to X$ over $C$ and an $m$-section $\phi(C')\equiv mC_1+\deg\phi^*\Oo_X(C_1)\cdot f=
mC_1+\deg S\cdot f=mC_1$\linebreak on $X$ for some $m\,|\,3$.
It contradicts the assumption that $S^3 E$ is not destabilized by a line subbundle.

Next, note that $\eta^*E\cong\sigma_*\Oo_B(C_1)$.
Because $B\sim 6C_1$ is effective, by pushing forward the exact sequence\linebreak
\vspace{-1em}
\[
0\to \Oo_X(-5C_1)\to \Oo_X(C_1)\to \Oo_B(C_1)\to 0
\]
on $X$ to $C$, we obtain an injection $E\to \xi_*\Oo_B(C_1)=\eta_*(\sigma_*\Oo_B(C_1))$.
Then it gives a nonzero morphism $\eta^*E\to\sigma_*\Oo_B(C_1)$ on $C'$, which is an isomorphism as $\eta^*E$ is stable and $\deg\eta^*E=\deg \Oo_B(C_1)=0$.

Then we check that $\Oo_B(2C_1)\in\Pr(B/C')$.
From $\det\sigma_*\Oo_B(C_1)\cong \det\sigma_*\Oo_{B}\otimes \Nm_{B/C'}(\Oo_B(C_1))$ and $(\det \sigma_*\Oo_B)^2=\Oo_{C'}$ as $\sigma$ is unramified,
we can observe that
\[
\Nm_{B/C'}(\Oo_B(2C_1))
=\{\Nm_{B/C'}(\Oo_B(C_1))\}^2
\cong (\det \sigma_*\Oo_B(C_1))^2.
\]
Since $\sigma_*\Oo_B(C_1)\cong \eta^*E$
 and $\det \eta^*E=\eta^*\det E=\Oo_{C'}$, we get $\Nm_{B/C'}(\Oo_B(2C_1))=\Oo_{C'}$. 
 
Also, we deduce that $\Oo_B(4C_1)=\Oo_B$ from (3).
Indeed, if $\Oo_B(C_1)=\Oo_B(3C_1)$, then $\Oo_B(2C_1)=\Oo_B$.
Otherwise, if $\Oo_B(C_1)\neq\Oo_B(3C_1)$, then $\eta^*E=\Oo_B(C_1)\oplus\Oo_B(3C_1)$ and $\Oo_B(4C_1)=\det \eta^*E=\Oo_B$.

Now we have $\Oo_B(2C_1)\in\Pr(B/C')$ and $\Oo_B(2C_1)\in J_2(B)$.
Thus $\sigma_*\Oo_B(2C_1)$ is orthogonal with values in $\Oo_{C'}$, and $\Oo_B(2C_1)\in \sigma^*J^0(C')$ by Proposition \ref{intersection}.
Due to Proposition \ref{img}, we can see that $\sigma_*\Oo_B(2C_1)$ is strictly semi-stable, and by Remark \ref{sss},
$\sigma_*\Oo_B(2C_1)=M\oplus N$ for some $M,\,N\in J_2(C')$.
Therefore, $S^2E\cong \eta_*M$ is obtained from the stability of $S^2 E$ and the following exact sequence from (4).
\[
0
\to S^2 E
\to \eta_*(\sigma_*\Oo_B(2C_1))\cong(\eta_*M)\oplus (\eta_*N)
\to S^2 E
\to 0
\]

Finally, $M\in \eta^*J^0(C)$ if and only if $S^2 E=\eta_*M$ is strictly semi-stable as in Example \ref{S3}.
\end{proof}

This is the end of the classification of $E\in\SU_C(2,\Oo_C)$ with stable $S^2 E$ and strictly semi-stable $S^3 E$; by Theorem \ref{classification_rank_two}, they are only given as in Example \ref{S3}, and the number of such $E$ is finite according to Proposition \ref{fundamental_theorem_of_square_tensor} and the argument preceding the proposition.
Recall that we observed in Theorem \ref{classification_rank_one} that if $S^3 E$ is destabilized by a line subbundle, then $S^2 E$ is necessarily not stable.

If $S^2 E$ is stable, then $S^3 E$ is not stable if and only if $S^3 E$ is destabilized by subbundles $E\otimes L^{\mp 1}\to S^3 E$ for some line bundle $L$ due to Proposition \ref{last}.
Moreover, from an exact sequence
\[
0
\to \Hom(L^{\mp 1}, S^4E)
\to \Hom(E\otimes L^{\mp 1}, S^3 E)
\to \Hom(L^{\mp 1}, S^2E)
\]
of Proposition \ref{tosmall} with $k=3$, we can see that $S^3 E$ is destabilized by subbundles $E\otimes L^{\mp 1}$ if and only if $S^4 E$ is destabilized by line subbundles $L^{\mp 1}\to S^4 E$ when $S^2 E$ is stable so that $\Hom(L^{\mp 1},S^2 E)=0$.
Therefore, we have the following corollary under the assumption on the stability of $S^2 E$.

\begin{Cor}
\label{clear_form}
If $S^2 E$ is stable, then $S^3 E$ is strictly semi-stable if and only if $S^4 E$ is destabilized by a line subbundle, and the number of such $E\in\SU_C(2,\Oo_C)$ is finite.
In particular, there exist only a finite number of ruled surfaces $X=\PP_C(E)$ for stable $E$ with even degree where $X$ contains a $4$-section of zero self-intersection whereas it has no such $m$-section for any $m<4$.
\end{Cor}

\section{Stability of Higher Symmetric Powers}

\subsection{Finiteness Result}
In the previous sections, we prove that the family of $E\in \SU_C(2,\Oo_C)$ with strictly semi-stable $S^2 E$ has positive dimension, but there are only finitely many $E\in\SU_C(2,\Oo_C)$ with strictly semi-stable $S^3 E$ outside this family.
We will strengthen the result as showing that there are only a finite number of $E\in\SU_C(2,\Oo_C)$ where $S^2 E$ is stable but $S^k E$ is not stable for some $k\geq 3$.

Throughout this and the next subsection, we assume that
\begin{itemize}
\item $E$ is a stable vector bundle on $C$ of rank~$2$ with trivial determinant,
\item $X=\PP_C(E)$ is the ruled surface $\Pi:\PP_C(E)\to C$,
\item $C_1$ is a unisecant divisor on $X$ satisfying $\Pi_*\Oo_X(C_1)=E$, and
\item $\pi=\Pi|_D: D\to C$ is the $k$-covering induced from a $k$-section $D$ on $X$.
\end{itemize}
Then, by Remark \ref{interesting}, $\pi$ is unramified if and only if $D^2=0$.
Also, from the exact sequence \eqref{basic} in the proof of Proposition \ref{tosmall} and its dual sequence with the self-duality $E^\vee\cong E$, we obtain two injections
\[
S^k E\to S^{k+1} E\otimes E
\quad\text{and}\quad
S^k E\to S^{k-1} E\otimes E.
\]
Thus if there is a subbundle $V\to S^k E$, then it induces injections $V\to S^{k+1}E\otimes E$ and $V\to S^{k-1}E\otimes E$, which respectively induce nonzero morphisms
\[\label{str}
V\otimes E\to S^{k-1}E
\quad\text{and}\quad
V\otimes E\to S^{k+1}E
\tag{4.1}\]
due to $\Hom(V\otimes E, S^{k\pm 1}E)
=H^0((V\otimes E)^\vee\otimes S^{k\pm 1}E)
=H^0(V^\vee\otimes S^{k\pm 1}E\otimes E)
=\Hom(V,S^{k\pm 1}E\otimes E)$.

\begin{Prop}\label{reduction}
Let $k\geq 2$.
If $S^{k-1}E$ is stable, then $S^k E$ is not destabilized by a subbundle $V\to S^k E$ of $\rk V<\frac{k}{2}$.
Moreover, when $S^{k-1}E$ is stable, $S^k E$ is not stable if and only if it is destabilized by\linebreak a stable subbundle $V\to S^k E$ of $\rk V=\frac{k+1}{2}$ if $k$ is odd or $\rk V=\frac{k}{2}$ if $k$ is even.
\end{Prop}

\begin{proof}
Assume that $S^k E$ is destabilized by a subbundle $V\to S^k E$ of $\rk V=r$ and $\deg V=0$.
As \eqref{str}, the injection $V\to S^k E$ gives a nonzero morphism $V\otimes E\to S^{k-1}E$.
Since $\deg (V\otimes E)=0$ and $S^{k-1}E$ is assumed to be stable, the morphism must be surjective, so $k=\rk S^{k-1}E
\leq \rk (V\otimes E)=2r$.
That is, $\frac{k+1}{2}\leq r$ if $k$ is odd or $\frac{k}{2}\leq r$ otherwise.

Due to Remark \ref{first}, when $S^k E$ is not stable, we can find a destabilizing subbundle $V\to S^k E$ of rank $r\leq \frac{k+1}{2}$ if $k$ is odd or $r\leq \frac{k}{2}$ otherwise, and each equality holds if $S^{k-1}E$ is stable as shown above.\linebreak
Because a destabilizing subbundle of $V$ destabilizes $S^k E$ in a smaller rank if it exists, $V$ must be stable.
\end{proof}

\begin{Thm}\label{higher}
Let $k\geq 2$.
If $S^{k-1}E$ is stable but $S^k E$ is not stable, then $E$ corresponds to one of the following cases.
\begin{itemize}
\item[(1)] $S^2 E$ is destabilized by a subbundle of rank $1$
\item[(2)] $S^3 E$ is destabilized by a subbundle of rank $2$
\item[(3)] $S^4 E$ is destabilized by a subbundle of rank $2$
\item[(4)] $S^6 E$ is destabilized by a subbundle of rank $3$
\end{itemize}
In particular, $S^k E$ is stable for every $k\geq 2$ if $S^m E$ is stable for all $m\leq 6$.
\end{Thm}

\begin{proof}
From the proof of Proposition \ref{reduction}, we observe that if $S^{k-1} E$ is stable but $S^k E$ is not stable, then there exist an injection $V\to S^k E$ and a surjection $V\otimes E\to S^{k-1}E$ for some stable vector bundle $V$ of degree $0$ and rank $r=\frac{k+1}{2}$ if $k$ is odd or $r=\frac{k}{2}$ otherwise.

If $k$ is odd, then $2r=k+1$, and the surjection $V\otimes E\to S^{k-1}E$ gives the exact sequence
\[
0\to L\to V\otimes E\to S^{k-1}E\to 0
\]
for some line bundle $L$ of degree $0$, and the injection $L\to V\otimes E$ induces a nonzero morphism $L\otimes E\cong L\otimes E^\vee\to V$.
This is possible only if $r=\rk V=2$ because $L\otimes E$ and $V$ are stable of the same degree.
Thus it implies that $k=3$.

Otherwise, if $k$ is even, then $2r=k$, so the surjection $V\otimes E\to S^{k-1}E$ becomes an isomorphism $V\otimes E\cong S^{k-1}E$.
On the other hand, the injection $V\to S^k E$ gives a nonzero morphism $V\otimes E\to S^{k+1}E$ as \eqref{str}.
By taking the dual, a nonzero morphism $S^{k+1}E\cong (S^{k+1}E)^\vee\to (V\otimes E)^\vee\cong (S^{k-1}E)^\vee\cong S^{k-1}E$ is obtained, which must be surjective due to the stability of $S^{k-1}E$.
Then, from the exact sequence
\[\label{S5}
0\to Q\to S^{k+1}E\to S^{k-1}E\to 0
\tag{4.2}
\vspace{-0.2em}
\]
for some vector bundle $Q$ of rank $2$ and degree $0$,
we can see that $S^{k+1}E$ is destabilized by a subbundle $Q\to S^{k+1}E$ of rank $2$.
Using the injections $S^{k+1}E\to S^k E\otimes E$ and $S^k E\to S^{k-1}E\otimes E$ twisted by $E$, the injection $Q\to S^{k+1}E$ induces an injection $Q\to S^{k-1}E\otimes E\otimes E$, and it yields a nonzero morphism $Q\otimes E\otimes E\to S^{k-1}E$ due to $\Hom(Q\otimes E\otimes E, S^{k-1}E)=\Hom(Q,S^{k-1}E\otimes E\otimes E)$ similar to \eqref{str}.
Since $S^{k-1}E$ is assumed to be stable, the morphism
\[
Q\oplus (Q\otimes S^2 E)
\cong Q\otimes (\Oo_C\oplus S^2 E)
\cong Q\otimes E\otimes E
\to S^{k-1}E
\]
is necessarily surjective, and again from the stability of $S^{k-1}E$, one of the morphisms $Q\to S^{k-1}E$ or $Q\otimes S^2 E\to S^{k-1}E$ must be surjective.
Therefore, $k=\rk S^{k-1}E\leq\max\{\rk Q,\,\rk(Q\otimes S^2E)\}=6$.
\end{proof}

Notice that cases (1) and (2) in the theorem are already studied in Section 2 and 3 respectively.
We will show that there are only a finite number of $E$ in cases (3) and (4) as in case (2).

\begin{Prop}\label{concluding}
Let $k\geq 3$.
Assume that there is a $k$-section $\pi: D\to C$ of zero self-intersection on $X$\linebreak and no such $m$-section for any $m<k$.
Then there is a surjection $\pi^*E\to R$ for some $R\in J_{2(k-1)(k-2)}(D)$.
\end{Prop}

\begin{proof}
Let $D\sim kC_1+\bb f$ for some $\bb\in\Pic(C)$, necessarily satisfying $\deg\bb=0$.
Due to Lemma \ref{integral} and Remark \ref{interesting}, $D$ is smooth and $\pi: D\to C$ is unramified.
Also, note that the morphism $\phi: D\to X$ over $C$ gives a surjection $\pi^*E\to R$ on $D$ for $R=\phi^*\Oo_X(C_1)=\Oo_D(C_1)$ with $\deg R=D.C_1=(kC_1+\bb f).C_1=0$.

By applying \cite[I:~p.~248]{Lazarsfeld04} and comparing the determinants after pushing forward the exact sequence
\[
0\to\Oo_X(-kC_1-\bb f)\to \Oo_X\to \Oo_D\to 0
\]
on $X$ to $C$, we obtain that $L^{-2(k-1)}=(\det(\Oo_C\oplus(S^{k-2}E\otimes L^{-1})))^2=(\det \pi_*\Oo_D)^2=\Oo_C$ for $L=\Oo_C(\bb)$ because $\pi$ is unramified.
Then, by Proposition \ref{k-2}, we have
\[
R^{k-2}=\Oo_D((k-2)C_1)
=\Oo_D(-\bb f)=\pi^*L^{-1},
\]
and hence $R^{2(k-1)(k-2)}
=(\pi^*L^{-1})^{2(k-1)}
=\pi^*L^{-2(k-1)}
=\pi^*\Oo_C
=\Oo_D$.
\end{proof}

\begin{Rmk}\label{hightor}
In the proof of Proposition \ref{concluding}, when $E$ has trivial determinant and $S^k E$ is destabilized by a quotient line bundle $S^k E\to L$, it is shown that $L^{2(k-1)}=\Oo_C$.
This is a generalization of Remark \ref{twotor}.\linebreak
It is also compatible with $L^4=\Oo_C$ for $k=3$ (Theorem \ref{classification_rank_one}) and $L^3=\Oo_C$ for $k=4$ (Proposition \ref{last}).
\end{Rmk}

\begin{Cor}\label{finite}
Let $k\geq 3$.
Then there exist at most finitely many $E\in\SU_C(2,\Oo_C)$ such that $S^k E$ is destabilized by a line subbundle but $S^m E$ is not destabilized by a line subbundle for any $m<k$.
\end{Cor}

\begin{proof}
The assertion follows from Proposition \ref{concluding} and the finiteness of the following data.
\begin{itemize}
\item the unramified $k$-coverings $\pi: D\to C$
\item the torsion line bundles on $D$ of order $2(k-1)(k-2)$
\item the direct summands of graded bundle of a Jordan-H\"{o}lder filtration associated to $\pi_*R$
\end{itemize}
By the adjoint property, a surjection $\pi^*E\to R$ gives a nonzero morphism $E\to \pi_*R$, so we can deduce that $E$ is a subbundle of $\pi_*R$ as $E$ is stable and $\deg E=\deg \pi_*R=0$.
\end{proof}

\begin{Lem}\label{generalization}
There exists the following exact sequence on $C$.
\[
0
\to S^{n-1}E\otimes S^{m-1}E
\to S^n E\otimes S^m E
\to S^{n+m}E
\to 0
\]
\end{Lem}

\begin{proof}
By restricting the natural morphism $\Pi^*\Pi_*\Oo_X(mC_1)\to \Oo_X(mC_1)$ to the fiber $Pf=\Pi^{-1}(P)$ of $X$,\linebreak
we have the morphism $\Pi^*\Pi_*\Oo_X(mC_1)|_{Pf}
\to\Oo_X(mC_1)|_{Pf}$ over each $P\in C$,
which is identified with the evaluation morphism $H^0(\PP^1,\Oo_{\PP^1}(mC_1))\otimes\Oo_{\PP^1}\to\Oo_{\PP^1}(m)$ from Grauert's theorem,
\[
\Pi^*\Pi_*\Oo_X(mC_1)|_{Pf}
\cong (\Pi_*\Oo_X(mC_1)\otimes \CC(P))\otimes \Oo_{Pf}
\cong H^0(Pf,\Oo_X(mC_1)|_{Pf})\otimes \Oo_{Pf}.
\]
As $\Oo_X(mC_1)|_{Pf}\cong\Oo_{\PP^1}(m)$ is generated by its global sections, the evaluation morphism is surjective over each $P\in C$, so the morphism $\Pi^*\Pi_*\Oo_X(mC_1)\to \Oo_X(mC_1)$ is surjective on $X$ by Nakayama's lemma.\linebreak
Thus we have the exact sequence
\[\label{gen}\tag{4.3}
0\to K\otimes \Oo_X(-C_1)\to \Pi^*\Pi_*\Oo_X(mC_1)\to \Oo_X(mC_1)\to 0
\]
for some vector bundle $K$ on $X$ of rank $m$ and $c_1(K)=\Oo_X$ which satisfies $K|_{Pf}\cong {\Oo_{Pf}}^{\oplus m}$ for each $P\in C$.\linebreak
Indeed, note that $K|_{Pf}\otimes\Oo_{\PP^1}(-1)\cong\Oo_{\PP^1}(a_1)\oplus\cdots\oplus \Oo_{\PP^1}(a_m)$ for some $a_i\leq 0$ with $a_1+\cdots+a_m=-m$, and from the associated long exact sequence of cohomology groups,
\[
0
\to H^0(\PP^1,K|_{Pf}\otimes\Oo_{\PP^1}(-1))
\to H^0(\PP^1,H^0(\PP^1,\Oo_{\PP^1}(mC_1))\otimes\Oo_{\PP^1})
\to H^0(\PP^1,\Oo_{\PP^1}(m)),
\]
we can observe that $H^0(K|_{Pf}\otimes\Oo_{\PP^1}(-1))
=H^0(\Oo_{\PP^1}(a_1))\oplus\cdots\oplus H^0(\Oo_{\PP^1}(a_m))=0$ because the last morphism is bijective.
Thus $a_i\leq -1$, and we have $a_i=-1$ for all $i=1,\,\ldots,\,m$.

By pushing forward exact sequence \eqref{gen} on $X$ to $C$ after twisting by $\Oo_X(-C_1)$, we have
\[
S^{m-1}E
\cong R^1\Pi_* (K\otimes \Oo_X(-2C_1))
\cong (\Pi_* K)^\vee
\]
from $\omega_{X/C}\cong\Oo_X(-2C_1)$ and the relative Serre duality.
Since $\Pi^*S^{m-1}E\cong \Pi^*(S^{m-1}E)^\vee\cong \Pi^*\Pi_*K$,\linebreak we get a morphism $\Pi^*S^{m-1}E\cong\Pi^*\Pi_*K\to K$ on $X$ where the latter morphism $\Pi^*\Pi_*K\to K$ is surjective\linebreak as $K|_{Pf}\cong {\Oo_{Pf}}^{\oplus m}$ is generated by its global sections for all $P\in C$.
Because $S^{m-1}E$ is of rank $m$, we have $K\cong \Pi^*S^{m-1}E$, and there exists the following exact sequence on $X$.
\vspace{-0.15em}
\[
0
\to \Pi^*S^{m-1}E\otimes \Oo_X(-C_1)
\to \Pi^*S^m E
\to \Oo_X(mC_1)
\to 0
\vspace{-0.15em}
\]
Therefore, we obtain the desired exact sequence by pushing forward the above sequence to $C$ after twisting by $\Oo_X(n C_1)$.
\end{proof}

\begin{Rmk}
The reviewer points out that as $E$ is the associated vector bundle of some principal $\mathrm{SL}_2\CC$-bundle $P\to C$, $S^kE$ is the associated vector bundle of $P(S^kV)$ where $V=\CC^2$ is the standard representation of $\mathrm{SL}_2\CC$, hence Lemma \ref{generalization} is a consequence of \cite[Exercise 11.11]{Fulton-Harris91}.
\end{Rmk}

\begin{Thm}\label{main}
If $S^2 E$ is stable, then every $S^k E$ is stable except for finitely many $E\in \SU_C(2,\Oo_C)$.
\end{Thm}

\begin{proof}
Assume that $S^k E$ is not stable and $S^m E$ is stable for all $m<k$.
By Theorem \ref{higher}, it remains to treat $k=3$, $4$, and $6$.
For $k=3$, we know from Corollary \ref{clear_form} that $S^4 E$ is destabilized by a line subbundle.
In the cases of even $k=4$, $6$, we can observe from exact sequence \eqref{S5} in the proof of Theorem \ref{higher} that there exists a surjection $S^{k+1}E\to S^{k-1} E$ so that $H^0(S^{k-1}E\otimes S^{k+1}E)=H^0((S^{k+1}E)^\vee\otimes S^{k-1}E)\neq 0$.
By taking global sections of the exact sequence of Lemma \ref{generalization} with $n={k-1}$ and $m={k+1}$,
\vspace{-0.2em}
\[
0\to S^{k-2} E\otimes S^k E\to S^{k-1} E\otimes S^{k+1} E\to S^{2k} E\to 0,
\vspace{-0.2em}
\]
we have either $H^0(S^{k-2}E\otimes S^k E)\neq 0$ or $H^0(S^{2k}E)\neq 0$ from $H^0(S^{k-1}E\otimes S^{k+1}E)\neq 0$.
If $H^0(S^{2k}E)= 0$, then $H^0(S^{k-2}E\otimes S^k E)\neq 0$, and in this case, by taking global sections of the exact sequence of the same Lemma with $n={k-2}$ and $m={k}$,
\vspace{-0.2em}
\[
0\to S^{k-3} E\otimes S^{k-1} E\to S^{k-2} E\otimes S^{k} E\to S^{2k-2} E\to 0,
\vspace{-0.2em}
\]
we have either $H^0(S^{k-3}E\otimes S^{k-1} E)\neq 0$ or $H^0(S^{2k-2}E)\neq 0$, where the former is impossible because, if then, it gives a nonzero morphism $S^{k-3}E\cong (S^{k-3}E)^\vee\to S^{k-1}E$ which destabilizes $S^{k-1}E$.
Therefore, from the assumption that $H^0(S^{k-1}E\otimes S^{k+1}E)\neq 0$, we obtain either $H^0(S^{2k}E)\neq 0$ or $H^0(S^{2k-2}E)\neq 0$.
In other words, either $S^{2k} E$ or $S^{2k-2}E$ is destabilized by a line subbundle.

Therefore, if $S^2 E$ is stable but $S^k E$ is not stable for some $k\geq 3$, then $S^l E$ is destabilized by a line subbundle for some $4\leq l\leq 12$.
However, since $S^2 E$ is stable, $S^3 E$ cannot be destabilized by a line subbundle due to Proposition \ref{tosmall}.
Thus $E$ is contained in the collection
\vspace{-0.3em}
\[
\bigcup_{l=4}^{12}
\left\{\text{$E\in\SU_C(2,\Oo_C)$}
\ |\ \text{$S^l E$ is destabilized by a line subbundle but $S^m E$ is not for any $m<l$}
\right\}
\vspace{-0.2em}
\]
which must be finite thanks to Corollary \ref{finite}.
\end{proof}

As the locus of strictly semi-stable $E\in \SU_C(2,\Oo_C)$ is given by the image of a map from the Jacobian variety, the locus is closed.
Also, the locus of $E\in \SU_C(2,\Oo_C)$ with strictly semi-stable $S^2 E$ is closed in $\SU_C(2,\Oo_C)$ because it is the union of finitely many images of maps from the Prym varieties.
Since a finite union of proper closed subsets is still a proper closed subset, the theorem gives the following fact.

\begin{Cor}\label{general}
If $E$ is stable, then every $S^k E$ is stable for general $E\in \SU_C(2,\Oo_C)$.
\end{Cor}

\subsection{Relation with \'{E}tale Triviality}
$E$ is said to be \emph{\'{e}tale-trivial} if there exists an unramified finite covering $\eta:C'\to C$ with $\eta^*E=\Oo_{C'}\oplus \Oo_{C'}$.
Thus if $E$ is \'{e}tale-trivial, then the trivialization induces a morphism $\phi: C'\to X$ over $C$ with $\deg \phi^*\Oo_X(C_1)=0$ and $\phi(C')\equiv kC_1+\deg\phi^*\Oo_X(C_1) f= kC_1$ for some $k>0$.
Therefore, the image $\phi(C')$ becomes a $k$-section of zero self-intersection on $X$ so that $S^k E$ is destabilized by a line subbundle.
We will see when the converse holds.

\newgeometry{top=85pt,bottom=80pt,left=80pt,right=80pt,footskip=30pt}
\begin{Lem}\label{pull-back}
Assume that $E$ splits over an unramified finite covering $\pi: D\to C$ as $\pi^*E=R^{-1}\oplus R$ for some line bundle $R$ on $D$.
Then $E$ is \'{e}tale-trivial if and only if $R\in J_m(D)$ for some $m>0$.
\end{Lem}

\begin{proof}
Recall that we assumed $\det E=\Oo_C$.
So if $\pi^*E$ splits, then it must     be of the form $\pi^*E=R^{-1}\oplus R$.
If $E$ is \'{e}tale-trivial, then there exists an unramified finite covering $\eta: C'\to C$ over which $E$ is trivialized.
Let $D'=D\times_C C'$ be the covering of $C$ which gives the following commutative diagram.
\vspace{-0.5em}
\[
\xymatrix @R=2pc{
D' \ar[r]^{\eta'} \ar[d]_{\pi'} &
D \ar[d]^\pi \\
C' \ar[r]^\eta &
C
}\]
From
\[
\eta'^*(R^{-1}\oplus R)
=(\pi\circ\eta')^*E
=(\eta\circ\pi')^*E
=\pi'^*(\Oo_{C'}\oplus \Oo_{C'})=\Oo_{D'}\oplus \Oo_{D'},
\]
we can observe that $\eta'^* R=\Oo_{D'}$.
It is possible only if $R$ is a torsion line bundle \cite[Proposition 11.4.3]{Birkenhake-Lange04}.

The converse holds by composing $\pi$ with a cyclic covering over which $R$ is trivialized.
\end{proof}

If $E$ is stable but $S^2 E$ is not stable, then, by Proposition \ref{dim}, there exists a nontrivial unramified double covering $\pi: B\to C$ and $E=\pi_*R$ for some $R\in J^0(B)$.
If $R\neq \iota^*R$, then $\pi^*E$ splits as $\pi^*E=\iota^*R\oplus R$
because $\pi^*E$ is invariant under the involution $\iota: B\to B$ induced by $\pi$.
Otherwise, if $R=\iota^*R$, then $R\in \pi^*J^0(C)$, so $E$ already splits on $C$ by Proposition \ref{img}.
Hence $\pi^*E$ splits over an unramified double covering $\pi:B\to C$ whenever $S^2 E$ is not stable.
Since $\det E=\Oo_C$, it must be of the form $\pi^*E=R^{-1}\oplus R$.
Therefore, we have the following observation applying the previous lemma.

\begin{Rmk}
If $E$ is stable but $S^2 E$ is not stable, then $E$ is \'{e}tale-trivial if and only if $E=\pi_*R$ for some nontrivial unramified double covering $\pi: B\to C$ and $R\in J_m(B)$ for some $m>0$.
\end{Rmk}

\begin{Lem}\label{k-2_twist}
Let $k\geq 3$. Assume that $S^k E$ is destabilized by a line subbundle $L^{-1}\to S^k E$ but $S^m E$ is not destabilized by a line subbundle for any $m< k$.
Then $S^{k-2} E\otimes L^{-1}\cong S^{k-2} E\otimes L$.
\end{Lem}

\begin{proof}
Let $D\sim kC_1+\bb f$ be the $k$-section on $X$ corresponding to the line subbundle $L^{-1}\to S^k E$.
Then $D$ is irreducible and reduced by Lemma \ref{integral}.
By pushing forward the exact sequences
\begin{gather*}
0\to \Oo_X(-kC_1-\bb f)\to \Oo_X\to \Oo_D\to 0,\\
0\to \Oo_X(-2C_1)\to \Oo_X((k-2)C_1+\bb f)\to \Oo_D((k-2)C_1+\bb f)\to 0
\end{gather*}
on $X$ to $C$ and using $\Oo_D((k-2)C_1+\bb f)=\Oo_D$ from Proposition \ref{k-2}, we get the following exact sequences.
\begin{gather*}
0\to \Oo_C\to \pi_*\Oo_D\to S^{k-2}E\otimes L^{-1}\to 0\\
0\to S^{k-2}E\otimes L\to \pi_*\Oo_D\to \Oo_C\to 0
\end{gather*}
Since $D$ is smooth by Remark \ref{interesting}, $\pi_*\Oo_D$ has a splitting of the injection $\Oo_C\to\pi_*\Oo_D$ \cite[I:~p.~248]{Lazarsfeld04}, which indeed follows from the hypothesis on $S^m E$ for $m<k$.
Therefore, $\Oo_C\oplus (S^ {k-2} E\otimes L^{-1})\cong \pi_*\Oo_D\cong \Oo_C\oplus (S^ {k-2} E\otimes L)$, and hence $S^{k-2} E\otimes L^{-1}\cong S^{k-2} E\otimes L$.
\end{proof}

\begin{Thm}\label{concludingconcluding}
Let $k\geq 3$.
If there is a $k$-section $\pi: D\to C$ of zero self-intersection on $X=\PP_C(E)$ and no such $m$-section for any $m<k$, then $\pi^*E=R^{-1}\oplus R$ for some $R\in J_{2(k-1)(k-2)}(D)$.
Moreover, if $S^2 E$ is stable, then $E$ is \'{e}tale-trivial if and only if $S^k E$ is destabilized by a line subbundle for some $k\geq 3$.
\end{Thm}

\begin{proof}
Let $L^{-1}\to S^k E$ be the destabilizing line subbundle corresponding to the $k$-section $D\sim kC_1+\bb f$ on $X$ for $L=\Oo_C(\bb)$.
Then $D$ is irreducible and reduced by Lemma \ref{integral}, and so $\pi:D\to C$ is unramified by Remark \ref{interesting}.
Let $R=\Oo_D(C_1)$.
We observe in Proposition \ref{concluding} that $R$ is a torsion line bundle and there exists a surjection $\pi^*E\to R$.
In order to conclude that $\pi^*E=R^{-1}\oplus R$, it remains to show that there exists another surjection $\pi^*E\to R^{-1}$ and $R^{-1}\neq R$.

We first claim that $R^{-1}\neq R$.
Suppose $\Oo_D(2C_1)=R^2=\Oo_D$.
By pushing forward the exact sequence
\[
0\to \Oo_X(-(k+4)C_1-\bb f)\to \Oo_X(-4C_1)\to \Oo_D(-4C_1)\to 0
\]
on $X$ to $C$, we have an injection $\pi_*\Oo_D=\pi_*\Oo_D(-4C_1)\to R^1\pi_*\Oo_X(-(k+4)C_1-\bb f)= S^{k+2} E\otimes L^{-1}$.
Since $H^0(\pi_*\Oo_D)\neq 0$ and  $H^0(\pi_*\Oo_D)\subseteq H^0(S^{k+2}E\otimes L^{-1})$, we get $H^0(S^{k+2}E\otimes L^{-1})\neq 0$.
So there exists a $(k+2)$-section $B\sim (k+2)C_1-\bb f$, and we have the following exact sequence on $X$.
\[
0\to \Oo_X(-(k+2)C_1+\bb f)\to \Oo_X\to \Oo_B\to 0
\]
By pushing forward the sequence to $C$, we obtain the isomorphism $\pi_*\Oo_B\cong \Oo_C\oplus (S^k E\otimes L)$ thanks to \cite[I:~p.~248]{Lazarsfeld04} if we suppose that $B$ is irreducible and reduced so that $B$ is smooth by Remark \ref{interesting}.
But
\[
h^0(\pi_*\Oo_B)=h^0(\Oo_C)+ h^0(S^k E\otimes L)\geq 2
\]
cannot hold if $B$ is irreducible and reduced, hence $B$ must be neither irreducible nor reduced.
So there exists an $m$-section of zero self-intersection on $X$ with $m\leq \frac{k+2}{2}$ due to Lemma \ref{integral}, and it contradicts the minimality of $k$.
Therefore, $R^2\neq \Oo_C$.

We next prove the existence of a surjection $\pi^*E\to R^{-1}$.
Because both $\pi^*E$ and $R^{-1}$ have degree $0$ and $\pi^*E$ is semi-stable, it suffices to find a nonzero morphism $\pi^*E\to R^{-1}$.
By the adjoint property, it is equivalent to show that $\Hom(E,\pi_*R^{-1})
\neq 0$.
By pushing forward the exact sequence
\[
0\to \Oo_X(-(k+1)C_1-\bb f)\to \Oo_X(-C_1)\to \Oo_D(-C_1)\to 0
\]
on $X$ to $C$, we can see that 
\[
\pi_*R^{-1}
\cong R^1\pi_*\Oo_X(-(k+1)C_1-\bb f)
\cong S^{k-1}E\otimes L^{-1}.
\]
Thus $\Hom(E,\pi_*R^{-1})\neq 0$ follows once we prove that $\Hom(E,S^{k-1}E\otimes L^{-1})\neq 0$.

For $k=3$, if $S^3 E$ is destabilized by a line subbundle $L^{-1}\to S^3 E$, then it induces a destabilizing subbundle $E\otimes L^{-1}\to S^2 E$ with the quotient bundle $S^2 E\to L^2$, so there is an isomorphism $E\cong E\otimes L^2$.
Hence we have $\Hom(E,S^2E\otimes L^{-1})=\Hom(E\otimes L^2,S^2E\otimes L^{-1})=\Hom(E\otimes L^{-1},S^2 E)\neq 0$ where the second equality is obtained from twisting by $L^{-3}$ and using the fact $L^2\in J_2(C)$ (see Proposition \ref{S3}).

For $k=4$, if $S^4 E$ has a destabilizing subbundle $L^{-1}\to S^4 E$, then it gives a subbundle $E\otimes L^{-1}\to S^3 E$ which destabilizes $S^3 E$, and we know from Proposition \ref{last} that there is also a subbundle $E\otimes L\to S^3 E$.
That is, $\Hom(E,S^3 E\otimes L^{-1})\neq 0$.

For the remaining $k\geq 5$, we start with the three exact sequences
\begin{gather*}
0\to S^{k-2}E\otimes A
\to E\otimes S^{k-1}E\otimes A
\to S^k E\otimes A\to 0, \\
0\to E\otimes S^{k-3}E\otimes A
\to E\otimes E\otimes S^{k-2}E\otimes A
\to S^{k-1} E\otimes A\to 0, \\
0\to S^{k-4}E \otimes A
\to E\otimes S^{k-3} E \otimes A
\to S^{k-2} E\otimes A\to 0
\end{gather*}
from \eqref{basic} being twisted by $A$ and $E\otimes A$ for a line bundle $A$ on $C$.
Thus we have
\begin{align*}
H^0(E\otimes S^{k-1} E\otimes A)\neq 0
\ &\Leftrightarrow\ 
H^0(S^k E\otimes A)\neq 0\ \text{or}\ H^0(S^{k-2}E\otimes A)\neq 0, \\
H^0(E\otimes E\otimes S^{k-2} E\otimes A)\neq 0
\ &\Leftrightarrow\ 
H^0(E\otimes S^{k-1} E\otimes A)\neq 0\ \text{or}\ H^0(E\otimes S^{k-3}E\otimes A)\neq 0, \\
H^0(E\otimes S^{k-3} E\otimes A)\neq 0
\ &\Leftrightarrow\ 
H^0(S^{k-2} E\otimes A)\neq 0\ \text{or}\ H^0(S^{k-4}E\otimes A)\neq 0.
\end{align*}
Then $H^0(S^k E\otimes L^{+1})=\Hom(L^{-1},S^k E)\neq 0$ implies $\Hom(E\otimes L,S^{k-1}E)=H^0(E\otimes S^{k-1}E\otimes L^{-1})\neq 0$ from the following implication diagram whose equivalence at the middle is established by Lemma~\ref{k-2_twist}.
An arrow labeled with $\times$ indicates an implication which leads to a contradiction to the assumption that there is no $m$-section of zero self-intersection on $X$ for any $m<k$.
\[
\xymatrix @C=-3.5pc @R=0pc{
H^0(S^kE\otimes L^{+1})\neq 0 \ar@{=>}[rdd] & &
H^0(S^{k-2}E\otimes L^{+1})\neq 0 & &
H^0(S^{k-4}E\otimes L^{+1})\neq 0\\ \\
& H^0(E\otimes S^{k-1} E\otimes L^{+1})\neq 0 \ar@{=>}[rdd] \ar[ruu]|\times & &
H^0(E\otimes S^{k-3} E\otimes L^{+1})\neq 0 \ar[luu]|\times \ar[ruu]|\times & \\ \\
& & H^0(E\otimes E\otimes S^{k-2}E\otimes L^{+1})\neq 0 \ar@{<=>}[ddd] \ar[ruu]|\times & & \\ \\ \\
& & H^0(E\otimes E\otimes S^{k-2}E\otimes L^{-1})\neq 0 \ar@{=>}[ldd] \ar[rdd]|\times & & \\ \\
& H^0(E\otimes S^{k-1}E\otimes L^{-1})\neq 0 \ar@{=>}[ldd] \ar[rdd]|\times & &
H^0(E\otimes S^{k-3}E\otimes L^{-1})\neq 0 \ar[ldd]|\times \ar[rdd]|\times \\ \\
H^0(S^kE\otimes L^{-1})\neq 0 & &
H^0(S^{k-2}E\otimes L^{-1})\neq 0 & &
H^0(S^{k-4}E\otimes L^{-1})\neq 0
}\]

Thus, for $k\geq 3$, $\pi^*E=R^{-1}\oplus R$ for some torsion line bundle $R$ over the unramified $k$-covering $\pi$.
Finally, we complete the proof applying Lemma \ref{pull-back}.
\end{proof}

\begin{Rmk}\label{etale}
Nori \cite[p.~35]{Nori76} defines a vector bundle $V$ on a complete, connected, reduced scheme $X$ to be \emph{finite} if there exists a finite collection $\mathcal{S}$ of vector bundles on $X$ such that for each $n\geq 0$,
\[
V^{\otimes n}=W_1\oplus W_2\oplus \cdots \oplus W_{k_n}\ \ 
\text{for some $W_i\in \mathcal{S}$}.
\]
It is also proven in the same paper that $V$ is finite if and only if $V$ is \emph{\'{e}tale-trivial} (over the base field of characteristic $0$).
Therefore, in the case where $X$ is a smooth projective curve $C$ of genus $g\geq 2$ and\linebreak $V$ is a vector bundle $E$ of rank $2$ with trivial determinant, Theorem \ref{concludingconcluding} says that when $S^2 E$ is stable, $E$ is finite if and only if $S^k E$ is destabilized by a line subbundle for some $k\geq 3$.

If $S^k E$ is not stable for some $k\geq 3$ which may assumed to be minimal, then $k=3$ or $4$ or $6$ due to Theorem~\ref{higher}.
Moreover, in those cases, from Proposition \ref{last} (with Proposition \ref{tosmall}) for $k=3$, and from the proof of Theorem \ref{main} for $k=4$ and $6$, we know that $S^l E$ is destabilized by a line subbundle for some $l\geq k$, and hence $E$ becomes finite by Theorem \ref{concludingconcluding}.
As a conclusion, when $S^2 E$ is stable, we can state that $E$ is finite if and only if $S^k E$ is not stable for some $k\geq 3$.
\end{Rmk}

\end{document}